\documentclass[a4paper,11pt]{article}

\usepackage{verbatim}
\usepackage[T1]{fontenc}
\usepackage{lmodern}
\usepackage[latin1]{inputenc}
\usepackage{setspace}
\usepackage{indentfirst}
\usepackage{bbm}
\usepackage{geometry}
\usepackage{hyperref}

\usepackage{times}
\usepackage{amsthm}
\usepackage{amsmath}
\usepackage{amsfonts}
\usepackage{amssymb}
\usepackage{bibentry}
\usepackage{color}
\usepackage{mathrsfs}
\usepackage{breqn}
\usepackage{stmaryrd}
\SetSymbolFont{stmry}{bold}{U}{stmry}{m}{n}

\usepackage{bm}
\usepackage{tikz-cd}
\usepackage{tikz}
\usetikzlibrary{matrix,arrows,decorations.pathmorphing}
\usepackage{graphicx}
\usepackage{subfigure}
\usepackage[toc,page]{appendix}

\newcommand{\Vect}{\mathsf{Vect}}

\DeclareMathOperator{\End}{End}

\DeclareMathOperator{\Ad}{Ad}

\DeclareMathOperator{\pr}{pr}

\DeclareMathOperator{\aff}{aff}

\newtheorem{Thm}{Theorem}[section]

\newtheorem{Pro}[Thm]{Proposition}
\newtheorem{Lem}[Thm]{Lemma}
\newtheorem{Cor}[Thm]{Corollary}
\newtheorem{Def-Pro}[Thm]{Definition-Proposition}
\newtheorem{Def}[Thm]{Definition}

\theoremstyle{definition}
\newtheorem{Ex}[Thm]{Example}
\newtheorem{Rm}[Thm]{Remark}

\begin{document}
\title{Affine structures on Lie groupoids}

\author{ \vspace{2mm}  Honglei Lang  \, Zhangju Liu  \, and \, Yunhe Sheng\\ \vspace{2mm}
\it{Department of Applied Mathematics, China Agricultural University, Beijing, 100873, China}\\ \vspace{2mm}
\it{Department of Mathematics, Peking University, Beijing, 100871, China}\\ \vspace{2mm}
\it{Department of Mathematics, Jilin University, Changchun, 130012, China}\\ \vspace{2mm}
hllang@cau.edu.cn,~~liuzj@pku.edu.cn,~~shengyh@jlu.edu.cn}

%\footnotetext{{\it{MSC}}:~}

\date{}
\maketitle

\makeatletter
\newif\if@borderstar
\def\bordermatrix{\@ifnextchar*{%
\@borderstartrue\@bordermatrix@i}{\@borderstarfalse\@bordermatrix@i*}%
}
\def\@bordermatrix@i*{\@ifnextchar[{\@bordermatrix@ii}{\@bordermatrix@ii[()]}}
\def\@bordermatrix@ii[#1]#2{%
\begingroup
\m@th\@tempdima8.75\p @\setbox\z@\vbox{%
\def\cr{\crcr\noalign{\kern 2\p@\global\let\cr\endline }}%
\ialign {$##$\hfil\kern 2\p@\kern\@tempdima & \thinspace %
\hfil $##$\hfil && \quad\hfil $##$\hfil\crcr\omit\strut %
\hfil\crcr\noalign{\kern -\baselineskip}#2\crcr\omit %
\strut\cr}}%
\setbox\tw@\vbox{\unvcopy\z@\global\setbox\@ne\lastbox}%
\setbox\tw@\hbox{\unhbox\@ne\unskip\global\setbox\@ne\lastbox}%
\setbox\tw@\hbox{%
$\kern\wd\@ne\kern -\@tempdima\left\@firstoftwo#1%
\if@borderstar\kern2pt\else\kern -\wd\@ne\fi%
\global\setbox\@ne\vbox{\box\@ne\if@borderstar\else\kern 2\p@\fi}%
\vcenter{\if@borderstar\else\kern -\ht\@ne\fi%
\unvbox\z@\kern-\if@borderstar2\fi\baselineskip}%
\if@borderstar\kern-2\@tempdima\kern2\p@\else\,\fi\right\@secondoftwo#1 $%
}\null \;\vbox{\kern\ht\@ne\box\tw@}%
\endgroup
}
\makeatother
\begin{abstract}
Affine structures on a Lie groupoid, including affine $k$-vector fields, $k$-forms and $(p,q)$-tensors are studied. We show that the space of affine structures is  a $2$-vector space over the space of multiplicative structures. Moreover, the space of affine multivector fields has a natural graded strict Lie $2$-algebra structure and affine $(1,1)$-tensors constitute a strict monoidal category. Such higher structures can be seen as the categorification of multiplicative structures on a Lie groupoid.
\end{abstract}

\tableofcontents

\section{Introduction}
Geometric structures on a Lie groupoid that are compatible with the groupoid multiplication are called multiplicative structures. They have been studied intensively these years and their infinitesimal correspondings are also developed. See  \cite{Xu, MX, Xu28} and \cite{BC2,BC,C} for multiplicative vector fields and multiplicative forms respectively and see \cite{BD} for the theory of multiplicative tensors. Beyond this, there are also multiplicative Dirac structures \cite{O}, multiplicative generalized complex structures \cite{JSX}, multiplicative contact and Jacobi structures \cite{C2,C3,IM}, multiplicative distributions \cite{JL} and multiplicative Manin pairs \cite{LB}, etc.
We refer to \cite{KS} for a survey on this subject.

 Our purpose is to study geometric structures that are compatible with the affinoid structure on a Lie groupoid.  It is motivated by Weinstein's work \cite{Weinstein}, where he studied Poisson manifolds carrying also affinoid structures. An affinoid structure on a space $X$ is a subset of $X^4$ whose elements are called parallelograms, with axioms modeled on the properties of the quaternary relation $\{(g,h,l,k);hg^{-1}=kl^{-1}\}$ on a group or a groupoid. Groupoids are affinoid spaces, but not every affinoid space arises in this way.  In \cite{M1,M2}, Mackenzie regarded affinoid structures as a type of double groupoids. He gave the equivalence  of affinoid structures, butterfly diagrams and generalized principal bundles and studied their infinitesimal invariants.

The multiplicativity condition for a $k$-vector field  (a $k$-form) on a Lie groupoid is known as the graph $\{(g,h,gh);s(g)=t(h)\}$ of the groupoid multiplication being a coisotropic (an isotropic) submanifold of $\mathcal{G}\times \mathcal{G}\times \mathcal{G}$.
 %with respect to $\Pi\oplus \Pi\oplus (-1)^{k+1} \Pi$
 %$\Pi\in \mathfrak{X}^k(\mathcal{G})$,
 %of $\mathcal{G}\times \mathcal{G}\times \mathcal{G}$ relative to $\Theta\oplus \Theta\oplus -\Theta$
 %$\Theta\in \Omega^k(\mathcal{G})$.
 A Lie groupoid $\mathcal{G}$ carries an affinoid structure with the set of parallelograms given by   $\{(g,h,l,hg^{-1}l)\}$ when $hg^{-1}l$ is well-defined.
 So the affine condition is naturally defined to be that the set of parallelograms is a coisotropic or an isotropic submanifold of $\mathcal{G}\times \mathcal{G}\times \mathcal{G}\times \mathcal{G}$ for $k$-vector fields or $k$-forms respectively. This gives the notions  of affine $k$-vector fields and affine $k$-forms on a Lie groupoid.
 This topic was first studied in \cite{Weinstein} by Weinstein. Then in \cite{Lu} Lu studies the dressing transformation, Poisson cohomology and also the symplectic groupoids of affine Poisson structures on Lie groups. For more information on affine Poisson structures, see also  \cite{B, DS, DLSW, U}.
To define affine $(p,q)$-tensors on a Lie groupoid $\mathcal{G}$,  we consider the tangent and cotangent groupoids of $\mathcal{G}$. A   $(p,q)$-tensor on $\mathcal{G}$ can be viewed as a function on the Lie groupoid $\Gamma:=\oplus^q T\mathcal{G} \oplus^p T^*\mathcal{G}\rightrightarrows \oplus^q TM\oplus^p A^*$. Then   a $(p,q)$-tensor is said to be affine if it is an affine function ($0$-form) on the Lie groupoid $\Gamma$. This definition coincides with the previous definitions for   affine $k$-vector fields and affine $k$-forms.

We shall first make clear the relations between affine and multiplicative structures. For Lie groups, Lu obtained  two multiplicative bivector fields  from an affine bivector field by using the right and left translations \cite{Lu}. We generalize this result to the case of Lie groupoids and obtain two multiplicative $k$-vector fields ($k$-forms,   $(p,q)$-tensors) from an affine $k$-vector field ($k$-form,   $(p,q)$-tensor).  Furthermore, we show that the space of affine structures is a 2-vector space over the vector space of multiplicative structures. Thus affine structures can be viewed as the categorification of multiplicative structures and affine structures define an equivalence relation on multiplicative structures.  For some cases, multiplicative structures are functors, as morphisms of Lie groupoids, then affine structures are natural transformations between these multiplicative structures. Moreover, for affine $k$-vector fields, the Schouten bracket gives rise to a  graded strict Lie $2$-algebra structure on the aforementioned 2-vector space. This recovers the strict Lie $2$-algebra structure on $1$-vector fields in \cite{BL} and is equivalent to the graded Lie $2$-algebra in \cite{BCLX}. We give the geometric support of this graded Lie $2$-algebra structure.
For affine $(1,1)$-tensors,  the composition of  affine $(1,1)$-tensors   defines a strict monoidal category structure on the aforementioned $2$-vector space.

We remark that on Lie groups,   affine $(p,q)$-tensors and multiplicative $(p,q)$-tensors are the same when $q\neq 0$. In particular,  the affine $k$-forms and multiplicative $k$-forms are the same.   An affine $k$-vector field differs from a multiplicative $k$-vector field by an element in $\wedge^k \mathfrak{g}$, where $\mathfrak{g}$ is the Lie algebra of the Lie group. Affine   $k$-vector fields, affine $k$-forms and affine $(1,1)$-tensors on pair groupoids are also analyzed in detail.

The organization of this paper is as follows. In Section $2$, we recall Lie 2-algebras, monoidal categories, tangent and cotangent Lie groupoids.  In Section $3$, we introduce the notion  of affine $k$-vector fields   and clarify the relation with multiplicative $k$-vector fields. We show that the space of affine $k$-vector fields is a $2$-vector space. Moreover, affine multivector fields are closed under the Schouten bracket, we thus get a graded strict Lie $2$-algebra structure on this 2-vector space. In Section $4$,  we introduce the notion  of  affine $k$-forms and study their properties analogously. In Section $5$,  we   define affine tensors on a Lie groupoid and  obtain a strict monoidal category structure on the space of affine $(1,1)$-tensors.

\section{Preliminary}
\subsection{Strict Lie $2$-algebras}
 Lie $2$-algebras are the  categorification of Lie algebras,   whose underlying spaces are $2$-vector spaces. See   \cite{Baez} for more details. Let $\Vect$ be the category of vector spaces.
\begin{Def}{\rm(\cite{Baez})}
 A $2$-vector space is a category in $\Vect$.
\end{Def}
Explicitly, a $2$-vector space is a category $V_1\rightrightarrows V_0$ whose spaces of objects and arrows are both vector spaces, such that the source and target maps $s,t:V_1\to V_0$, the identity-assigning map $\iota:V_0\hookrightarrow V_1$, and the composition $\circ: V_1\times_{V_0} V_1\to V_1$ are all linear.

A $2$-vector space is completely determined by the vector spaces $V_0, V_1$ with the source, target and the identity-assigning map. Actually, given $f\in V_1$, define its arrow part by $\overrightarrow{f}=f-\iota(s(f))$. Then $s(\overrightarrow{f})=0$ and $t(\overrightarrow{f})=t(f)-s(f)$.
So we can identity $f:x\to y$ with the pair $(x,\overrightarrow{f})$. With this notation, for $f:x\to y,g:y\to z$, their composition is defined as $g\circ f=(x,\overrightarrow{f}+\overrightarrow{g})$.
Any arrow $(x,\overrightarrow{f})$ has an inverse $(x+t(\overrightarrow{f}),-\overrightarrow{f})$, so a $2$-vector space is always a Lie groupoid.

A $2$-vector space is equivalent to a $2$-term chain complex of vector spaces. On the one hand, given a $2$-vector space $V_1\rightrightarrows V_0$, the corresponding $2$-term complex is $t: \ker s\to V_0$. On the other hand, given a chain complex $C_1\to C_0$, the $2$-vector space is $C_0\oplus C_1\to C_0$. We refer to \cite{Baez} for the details.

\begin{Def}{\rm(\cite{Baez})}
A strict Lie $2$-algebra is a $2$-vector space $V$ together with a skew-symmetric bilinear functor, the bracket, $[\cdot,\cdot]:V\times V\to V$ satisfying the Jacobi identity.
\end{Def}
A strict Lie $2$-algebra is equivalent to a strict $2$-term $L_\infty$-algebra. Namely, a $2$-term complex $d: C_1\to C_0$ with skew-symmetric brackets $[\cdot,\cdot]: C_0\times C_0\to C_0$ and $[\cdot,\cdot]:C_0\times C_1\to C_1$ satisfying the Jacobi identity and $[da,b]=[a,db], d[x,a]=[x,da]$ for $x\in C_0$ and $a,b\in C_1$. See \cite{Baez} for details.

When the spaces of objects and morphisms are graded vector spaces and the Lie bracket is a graded Lie bracket, we call it a {\bf graded strict Lie $2$-algebra}. We refer to \cite{BCLX} for the explicit definition.

%While for a general Lie $2$-algebra, the skew-symmetric bracket  is relaxed to satisfy the Jacobi identity up to isomorphisms. See \cite{Baez} for the equivalence of Lie $2$-algebras and $2$-term $L_\infty$-algebras.

\subsection{Strict monoidal categories}

A monoidal category is a category $\mathcal{C}$ with a bifunctor $\circ:\mathcal{C}\times \mathcal{C}\to \mathcal{C}$, its product, which is associative up to a natural isomorphism and has an object which is a left unit and a right unit for the product up to natural isomorphisms. For our purpose, we only consider the category with a product which is strict associative and has a strict two-sided identity object.
\begin{Def}{\rm(\cite{ML})}
A strict monoidal category $(\mathcal{C},\circ,e)$ is a category $\mathcal{C}$ with a bifunctor $\circ: \mathcal{C}\times \mathcal{C}\to \mathcal{C}$, which is associative:
\[\circ (\circ\times 1)=\circ (1\times \circ): \mathcal{C}\times \mathcal{C}\times \mathcal{C}\to \mathcal{C},\]
and with an object $e$ which is a left and right unit for $\circ$:
\[\circ (e\times 1)=\mathrm{id}_{\mathcal{C}}=\circ (1\times e):\mathcal{C}\to \mathcal{C}.\]
\end{Def}

The bifunctor $\circ$ here assigns to each pair of objects $x,y$ an object $x\circ y$ and each pair of arrows $f:x\to x', g:y\to y'$ an arrow $f\circ g:x\circ y\to x'\circ y'$. Thus $\circ$ being a bifunctor means that
\begin{eqnarray}\label{bifunctor}
1_x\circ 1_y=1_{x\circ y},\qquad (f'\circ g')\cdot (f\circ g)=(f'\cdot f)\circ (g'\cdot g),
\end{eqnarray}
whenever $f',f$ and $g',g$ are composable. Here $\cdot$ is the multiplication in the category $\mathcal{C}$. The associative law and the unit law in the definition hold both for objects and arrows.

%One example of strict monoidal category is a $2$-vector space with the addition in vector spaces as the product.

\subsection{Tangent and cotangent Lie groupoids}

We recall the definition of the tangent and cotangent Lie groupoids of a Lie groupoid.

For a Lie groupoid $\mathcal{G}\rightrightarrows M$, denote its source and target maps by $s, t:\mathcal{G}\to M$. Two elements $g,h\in \mathcal{G}$ are multiplicable or composable if and only if $s(g)=t(h)$ and their product is written as $g\cdot h$ or simply as $gh$.  Such a pair is called a {\bf multiplicable pair}. We denote the space of multiplicable pairs by $\mathcal{G}^{(2)}$.
Let $A$ be the Lie algebroid of $\mathcal{G}$. For $u\in \Gamma(A)$, its right and left translations $\overrightarrow{u},\overleftarrow{u}\in \mathfrak{X}(\mathcal{G})$ are defined by
\[\overrightarrow{u}(g)=dR_g(u_{t(g)}),\qquad \overleftarrow{u}(g)=-dL_g(d\mathrm{inv}(u_{s(g)}),\]
where $R_g,L_g$ are the right and left multiplications on $\mathcal{G}$ and $\mathrm{inv}:\mathcal{G}\to \mathcal{G}$ is the inverse map in $\mathcal{G}$.

Throughout this paper, by abuse of notations, we use the same notations $s$ and $t$ to denote the source and target of any Lie groupoid and we adopt the same notation to denote a map and its tangent map.

Given a  Lie groupoid $\mathcal{G}\rightrightarrows M$, its tangent bundle $T\mathcal{G}\rightrightarrows TM$ with the differential of the structure maps of $\mathcal{G}$ is again a Lie groupoid.

Its cotangent bundle $T^*\mathcal{G}$ is also equipped with a Lie groupoid structure which is over $A^*$, written as $T^*\mathcal{G}\rightrightarrows A^*$. First we have the inclusion $A^*\hookrightarrow T^*\mathcal{G}$ since $T^*\mathcal{G}|_M\cong T^*M\oplus A^*$.
The source and target maps of an element $\xi\in T^*_g\mathcal{G}$ with $g\in \mathcal{G}$ is
\begin{eqnarray*}\label{right-left}
\langle s(\xi), u\rangle=\langle\xi, \overleftarrow{u}\rangle,\qquad \langle t(\xi),v\rangle=\langle \xi, \overrightarrow{v}\rangle\qquad
\forall u\in A_{s(g)}, v\in A_{t(g)},
\end{eqnarray*}
So for any $u\in \Gamma(A)$, seen as a function on the base manifold $A^*$, we get the formulas
\begin{eqnarray}\label{s cot}
s^* u=\overleftarrow{u},\qquad t^*u=\overrightarrow{u},
\end{eqnarray}
For a multiplicable pair $(g,h)\in \mathcal{G}^{(2)}$, if $\xi\in T_g^*\mathcal{G}$ and $\eta\in T_h^*\mathcal{G}$ are multiplicable,  the product is the element $\xi\cdot\eta\in T^*_{gh}\mathcal{G}$ such that
\begin{eqnarray}\label{multi}
(\xi\cdot\eta)(X\cdot Y)=\xi(X)+\eta(Y),\qquad \forall (X,Y)\in T\mathcal{G}^{(2)},
\end{eqnarray}
where $X\cdot Y\in T_{gh}\mathcal{G}$ is the product of $X\in T_g \mathcal{G}$ and $Y\in T_h \mathcal{G}$ in the Lie groupoid $T\mathcal{G}$.
%for $X\in T_g \mathcal{G}$ and $Y\in T_h \mathcal{G}$ such that $ds(X)=dt(Y)$.
See for example \cite{KS, LL} for more explanation of the cotangent groupoid.

\section{Affine $k$-vector fields on a Lie groupoid}

An affinoid structure on a space $X$ is a subset of $X^4$ whose elements are seen as parallelograms, with axioms modeled on the properties of the quaternary relation $\{(g,h,l,k)|hg^{-1}=kl^{-1}\}$ on a group or a groupoid \cite{Weinstein}.
In particular, a groupoid has an affinoid structure with parallelograms given by the relation \[\{(g,h,l,hg^{-1}l);s(g)=s(h),t(g)=t(l)\}.\]
A $k$-vector field on a Lie groupoid is called affine when it is compatible with the affinoid structure in the sense that the submanifold of parallelograms  is coisotropic. While a $k$-vector field is multiplicative when the graph $\{(g,h,gh);s(g)=t(h)\}$ of the multiplication, or space of triangles, is coisotropic.

Let $V$ be a vector space and $\Pi\in \wedge^k V$. A subspace $W\subset V$ is coisotropic with respect to $\Pi$ if
\[\Pi(\xi_1,\cdots,\xi_k)=0,\qquad \forall \xi_1,\cdots,\xi_k\in W^0,\]
where $W^0$ is the annihilator space of $W$, namely, $W^0=\{\xi\in V^*;\xi(w)=0,\forall w\in W\}$.
More generally, for a manifold $M$   and $\Pi\in \mathfrak{X}^k(M)$, a submanifold $S$ of $M$ is   coisotropic with respect to $\Pi$ if $T_x S$ is coisotropic with respect to $\Pi_x$ for all $x\in S$.

The following definition is motivated by Weinstein's definition for  affine Poisson structures on a Lie groupoid \cite{Weinstein}.

\begin{Def}\label{aff vf}
A $k$-vector field $\Pi\in \mathfrak{X}^k(\mathcal{G})$ on a Lie groupoid $\mathcal{G}$ is called {\bf affine} if the submanifold \[S:=\{(g,h,l,hg^{-1}l); s(g)=s(h),t(g)=t(l)\}\subset \mathcal{G}\times \mathcal{G}\times \mathcal{G}\times \mathcal{G}\]  is coisotropic with respect to $\Pi\oplus (-1)^{k+1}\Pi\oplus (-1)^{k+1} \Pi\oplus \Pi$.
\end{Def}

%Explicitly, one has
%\[(\Pi\times (-1)^{k+1}\Pi\times (-1)^{k+1}\Pi\times \Pi)(\xi_1,\cdots,\xi_k)=0,\qquad \forall \xi_1,\cdots,\xi_k\in S^\perp,\]
%where \[S^\perp:=\{\xi\in T_\lambda^*(\mathcal{G}\times \mathcal{G}\times \mathcal{G}\times \mathcal{G})| \lambda\in S,\xi(v)=0,\forall v\in T_\lambda S\}.\]

Comparatively,  a $k$-vector field on a Lie groupoid $\mathcal{G}$ is multiplicative \cite{Xu,MX,Xu28} if it satisfies that the submanifold $\{(g,h,gh);s(g)=t(h)\}\subset \mathcal{G}\times \mathcal{G}\times \mathcal{G}$ is coisotropic relative to $\Pi\oplus \Pi\oplus (-1)^{k+1} \Pi$.

It is shown in
 \cite{Xu,Weinstein} that a $k$-vector field $\Pi\in \mathfrak{X}^k(\mathcal{G})$ is multiplicative if and only if it is affine and the base manifold $M$ is coisotropic with respect to $\Pi$.

As shown in \cite{Xu28} for the $k=2$ case, for any $\mu\in T^*_{gh} \mathcal{G}$, the covector $(-\mu, L_{\mathcal{X}}^*\mu, R_{\mathcal{Y}}^*\mu,-L_{\mathcal{X}}^*R_{\mathcal{Y}}^*\mu)$ is conormal to $S$, that is, an element in the annihilator space of $TS$ at $(gh,h,g,s(g))$, where $\mathcal{X}$ and $\mathcal{Y}$ are bisections passing through $g$ and $h$. Another two classes of vectors conormal to $S$ are  $(-t^*\eta,t^*\eta, 0,0)$ and $(-s^*\xi,s^*\xi,0,0)$ for $\eta\in T_{t(g)}^*M$ and  $\xi\in T_{s(g)}^* M$.  We thus get an explicit description of affine $k$-vector fields:
\begin{Lem}\label{affine vf}
A $k$-vector field $\Pi\in \mathfrak{X}^k(\mathcal{G})$ on a Lie groupoid $\mathcal{G}$ is affine if and only if the following two conditions hold:
\begin{itemize}
\item[\rm(1)]
for any $(g,h)\in \mathcal{G}^{(2)}$,
\begin{eqnarray}\label{affine}
\Pi(gh)=L_{\mathcal{X}} \Pi(h)+R_{\mathcal{Y}} \Pi(g)-L_{\mathcal{X}}\circ R_{\mathcal{Y}}(\Pi(s(g))),
\end{eqnarray}
where $\mathcal{X}$ and $\mathcal{Y}$ are any two local bisections passing through $g$ and $h$ respectively.
\item[\rm(2)] for any $\xi\in \Omega^1(M)$, $\iota_{t^*\xi} \Pi$ is right invariant.
\end{itemize}
\end{Lem}
%\begin{Rm}
%Multivector fields satisfying (\ref{affine}) are called   affine multivector fields  by Xu in \cite{Xu28}. Here we have an extra condition coming naturally from the definition of affine $k$-vector fields.
%\end{Rm}
An equivalent description of (\ref{affine}) is that
$[\Pi, \overrightarrow{X}]$ is right-invariant for $X\in \Gamma(A)$ \cite{MX}.

On the other hand, we shall see that one affine $k$-vector field defines two multiplicative $k$-vector fields.

 For a $k$-vector field $\Pi\in \mathfrak{X}^k(\mathcal{G})$,  its restriction on $M$ has $k+1$ components:
\[\Pi|_M\in \Gamma(\wedge^k T\mathcal{G}|_M)\cong\Gamma(\wedge^k (TM\oplus A))=\Gamma(\wedge^k TM\oplus \wedge^{k-1} TM\wedge A\oplus \cdots \oplus \wedge^k A).\]
We denote by $\pi$ the $\wedge^k A$-component:
\begin{eqnarray*}
\pi=\mathrm{pr}_{\wedge^k A} \Pi|_M\in \Gamma(\wedge^k A).
\end{eqnarray*}
So the base manifold $M$ is coisotropic with respect to $\Pi\in \mathfrak{X}^k(\mathcal{G})$ if $\pi=0$.
\begin{Pro}\label{leftright} Let $\Pi$ be a  $k$-vector  field on the Lie groupoid $\mathcal{G}$ with $\pi=\mathrm{pr}_{\wedge^k A} \Pi|_M$.
Define
\begin{eqnarray}\label{affmul}
\Pi_r=\Pi-\overrightarrow{\pi},\qquad \Pi_l=\Pi-\overleftarrow{\pi}.
\end{eqnarray}
Then $\Pi$ is affine if and only if $\Pi_l$  or $\Pi_r$ is a multiplicative $k$-vector field on $\mathcal{G}$. Here the right and left translations are
\[\overrightarrow{\pi}(g):=R_{g}(\pi_{t(g)}),\qquad \overleftarrow{\pi}(g):=-L_g(\mathrm{inv}(\pi_{s(g)})),\qquad \forall g\in \mathcal{G}.\]
\end{Pro}
\begin{proof}
It is known from \cite{MX} that a $k$-vector field $\Pi$ on $\mathcal{G}$ satisfying $\Pi(gh)=L_{\mathcal{X}} \Pi(h)+R_{\mathcal{Y}} \Pi(g)-L_{\mathcal{X}}\circ R_{\mathcal{Y}}(\Pi(s(g)))$ is equivalent to say  that $[\Pi,\overrightarrow{X}]$ is right-invariant for any $X\in \Gamma(A)$.

For any $X\in \Gamma(A)$,  $[\Pi_r,\overrightarrow{X}]$ is right-invariant if and only if $[\Pi,\overrightarrow{X}]$ is right-invariant. So $\Pi$ satisfies (\ref{affine}) if and only if $\Pi_r$ satisfies (\ref{affine}). Besides, since $t\circ R_g=t$, it is direct to see that $\iota_{t^* \xi} \overrightarrow{\pi}$ is right-invariant for $\xi\in \Omega^1(M)$. Also it is obvious that $M$ is coisotropic with respect to $\Pi_r$. We conclude that $\Pi_r$ is multiplicative if and only if $\Pi$ is affine.

For $\Pi_l$, by the fact that $[\Pi_l,\overrightarrow{X}]=[\Pi,\overrightarrow{X}]$, $\iota_{t^*\xi} \overleftarrow{\pi}=0$ and $M$ is coisotropic with respect to $\Pi_l$, we obtain the conclusion that $\Pi$ is affine if and only if $\Pi_l$ is multiplicative.
\end{proof}

\begin{Ex}
 Multiplicative $k$-vector fields on $\mathbbm{R}^n$ are linear $k$-vector fields. An Affine $k$-vector field is of the form \[\sum f^{i_1,\cdots,i_k}(x)\frac{\partial}{\partial x^{i_1}}\wedge\cdots \wedge \frac{\partial}{\partial x^{i_k}}+\sum c^{i_1,\cdots,i_k} \frac{\partial}{\partial x^{i_1}}\wedge\cdots \wedge \frac{\partial}{\partial x^{i_k}},\] where $f^{i_1,\cdots,i_k}(x)$ is a linear function on $\mathbbm{R}^n$ and $c^{i_1,\cdots,i_k}$ is a constant. Namely, an affine $k$-vector field is a sum of a linear $k$-vector field and a constant $k$-vector field.
\end{Ex}

\begin{Ex}
%It is known that the Poisson groupoid structures on the pair groupoid $M\times M\rightrightarrows M$ are of the form $(\Pi,-\Pi)$ with $\Pi$ being a Poisson structure on $M$. We claim the
Multiplicative $k$-vector fields on the pair groupoid $M\times M\rightrightarrows M$ all have the form
$(\Pi,-\Pi)$ for $\Pi\in \mathfrak{X}^k(M)$ and affine $k$-vector fields on $M\times M$ are of the form $(\Pi,\Pi')$ for two $k$-vector fields $\Pi,\Pi'\in \mathfrak{X}^k(M)$.
\end{Ex}

\begin{Ex}
Let $\mathcal{G}$ be a Lie groupoid with Lie algebroid $A$. For any $\pi\in \Gamma(\wedge^k A)$, the $k$-vector field $\Pi=\overrightarrow{\pi}$ is affine and the associated two multiplicative $k$-vector fields are $\Pi_r=0$ and $\Pi_l=\overrightarrow{\pi}-\overleftarrow{\pi}$.
\end{Ex}

The space $\mathfrak{X}^k_{\mathrm{aff}}(\mathcal{G})$ of affine $k$-vector fields is a vector space with the space $\mathfrak{X}^k_{\mathrm{mult}}(\mathcal{G})$ of multiplicative $k$-vector fields as a linear subspace.
\begin{Thm}\label{2-group for vf}
We have a $2$-vector space
\[\mathfrak{X}^k_{\mathrm{aff}}(\mathcal{G})\rightrightarrows \mathfrak{X}^k_{\mathrm{mult}}(\mathcal{G}),\] where the groupoid structure is as follows:
 the source and target maps are given by $s(\Pi)=\Pi_r$ and $t(\Pi)=\Pi_l$ as defined in \eqref{affmul}, and the multiplication $*$ is
\[\Pi* \Pi'=\Pi+\overleftarrow{\pi'},\]
for a pair $\Pi,\Pi'$ of affine $k$-vector fields such that $\Pi_r=\Pi'_l$. Here $\pi'=\mathrm{pr}_{\wedge^k A} \Pi'|_M$ is the $\wedge^k A$-component of $\Pi'|_M$.\end{Thm}
\begin{proof}
We first verify the groupoid structure. It follows from $\Pi_r=\Pi'_l$ that $\Pi-\overrightarrow{\pi}=\Pi'-\overleftarrow{\pi'}$. Then
\[s(\Pi* \Pi')=\Pi+\overleftarrow{\pi'}-\overrightarrow{\pi}-\overrightarrow{\pi'}=\Pi'-\overrightarrow{\pi'}=\Pi'_r=s(\Pi'),\]
and
\[t(\Pi* \Pi')=\Pi+\overleftarrow{\pi'}-\overleftarrow{\pi}-\overleftarrow{\pi'}=\Pi_l=t(\Pi).\]
Here we have used the fact that \[\mathrm{pr}_{\wedge^k A} \overleftarrow{\pi'}|_M=-\mathrm{pr}_{\wedge^k A}\mathrm{inv}(\pi')=\mathrm{pr}_{\wedge^k A} (\pi'-\rho(\pi'))=\pi'.\] Here if $\pi'=X_1\wedge\cdots \wedge X_k$, we have 
\[-\mathrm{inv}(\pi')=-\mathrm{inv}(X_1)\wedge \cdots \wedge (-\mathrm{inv}(X_k))=(X_1-\rho(X_1))\wedge\cdots (X_k-\rho(X_k)).\]
For the associativity of this multiplication, let $\Pi''$ be another affine $k$-vector field such that $\Pi'_r=\Pi''_l$. We see
\[(\Pi* \Pi')* \Pi''=\Pi+\overleftarrow{\pi'}+\overleftarrow{\pi''}=\Pi* (\Pi' *\Pi'').\]
Also, it is immediate that all the groupoid structures are linear. This gives a $2$-vector space structure on $\mathfrak{X}^k_{\aff}(\mathcal{G})$.
\end{proof}

 For $\Pi\in \mathfrak{X}^k_{\aff}(\mathcal{G})$,  its inverse in this $2$-vector space is
 \begin{eqnarray}\label{Pi inverse}
 \Pi^{-1}=\Pi-(\overrightarrow{\pi}+\overleftarrow{\pi}),\qquad \pi=\mathrm{pr}_{\wedge^k A} \Pi|_M.
 \end{eqnarray}
\begin{Rm}
In \cite{Lu}, Lu considered the case when $\Pi\in \mathfrak{X}^2(\mathcal{G})$ is an affine Poisson vector field on a Lie group. This affine vector field $\Pi^{-1}$ is also Poisson and is called the {\bf opposite affine Poisson structure} of $\Pi$.
 We see here that  it is actually the inverse of $\Pi$ in the $2$-vector space given above.
\end{Rm}

 \begin{Cor}
The associated $2$-term chain complex of vector spaces for the $2$-vector space in the above theorem is
\[\Gamma(\wedge^k A)\to \mathfrak{X}^k_{\mathrm{mult}}(\mathcal{G}),\qquad \pi\mapsto \overrightarrow{\pi}-\overleftarrow{\pi}.\]
 \end{Cor}

 In addition to this, since affine multivector fields are closed under the Schouten bracket \cite{Xu}, we further obtain a graded strict Lie $2$-algebra on this $2$-vector space. See \cite{Baez} for the details of Lie $2$-algebras.
 \begin{Thm}\label{Lie 2}
We have a graded strict Lie $2$-algebra structure on \[\oplus_k \mathfrak{X}^k_{\aff}(\mathcal{G})\rightrightarrows \oplus_k \mathfrak{X}^k_{\mathrm{mult}}(\mathcal{G})\]
where the bracket is the Schouten bracket.
% and the graded $2$-vector space structure is given in Theorem \ref{2-group for vf}.
\end{Thm}
\begin{proof}
The Schouten bracket defines a graded Lie algebra structure on $\oplus_k\mathfrak{X}^k_{\aff}(\mathcal{G})$. It suffices to check that it is a functor. Let $\Pi_1,\Pi'_1\in \mathfrak{X}^k_{\aff}(\mathcal{G})$ and $\Pi_2,\Pi'_2\in \mathfrak{X}^l_{\aff}(\mathcal{G})$ be two multiplicable pairs, that is $(\Pi_1)_r=(\Pi'_1)_l$ and $(\Pi_2)_r=(\Pi'_2)_l$. The Schouten bracket $[\cdot,\cdot]:\mathfrak{X}^k_{\aff}(\mathcal{G})\times \mathfrak{X}^l(\mathcal{G})\to \mathfrak{X}^{k+l}(\mathcal{G})$ being a functor means
\begin{eqnarray}\label{functor}
[(\Pi_1,\Pi_2)*(\Pi'_1,\Pi'_2)]=[\Pi_1,\Pi_2]*[\Pi'_1,\Pi'_2].
\end{eqnarray}
Actually, by Theorem \ref{2-group for vf}, the left hand side equals to
\[[\Pi_1+\overleftarrow{\pi'_1},\Pi_2+\overleftarrow{\pi'_2}]=[\Pi_1,\Pi_2]+[\Pi_1,\overleftarrow{\pi'_2}]+[\overleftarrow{\pi'_1},\Pi_2]-\overleftarrow{[\pi'_1,\pi'_2]},\]
where $\pi'_1=\mathrm{pr}_{\wedge^k A}{\Pi'_1}|_M$ and $\pi'_2=\mathrm{pr}_{\wedge^l A} \Pi'_2|_M$. And the right hand side amounts to
\[[\Pi_1,\Pi_2]+\overleftarrow{\mathrm{pr}_{\wedge^{k+l} A} [\Pi'_1,\Pi'_2]|_M}.\]
By strict calculation, we have \begin{eqnarray*}
 [\Pi'_1,\Pi'_2]&=&[(\Pi'_1)_l,(\Pi'_2)_l]+[\overleftarrow{\pi'_1},(\Pi'_2)_l]+[(\Pi'_1)_l,\overleftarrow{\pi'_2}]+[\overleftarrow{\pi'_1},\overleftarrow{\pi'_2}]\\ &=&[(\Pi'_1)_l,(\Pi'_2)_l]+[\overleftarrow{\pi'_1},\Pi_2]+[\Pi_1,\overleftarrow{\pi'_2}]+[\overleftarrow{\pi'_1},\overleftarrow{\pi'_2}],\end{eqnarray*} where we have used $(\Pi_1)_r=(\Pi'_1)_l$ and $(\Pi_2)_r=(\Pi'_2)_l$ and the fact that $[\overleftarrow{X},\overrightarrow{Y}]=0$ for any $X,Y\in \Gamma(A)$. Moreover,
$\Pi_1$ is multiplicative so that $[\Pi_1, \overleftarrow{\pi'_2}]$ is left-invariant and so does $[\overleftarrow{\pi'_1},\Pi_2]$. From this and $(\Pi'_1)_l$  ($(\Pi'_2)_l$) has no component in $\wedge^k A$ ($\wedge^l A$), we see
\[[\Pi_1,\Pi_2]+\overleftarrow{\mathrm{pr}_{\wedge^{k+l} A} [\Pi'_1,\Pi'_2]|_M}=[\Pi_1,\Pi_2]+[\overleftarrow{\pi'_1},\Pi_2]+[\Pi_1,\overleftarrow{\pi'_2}]-\overleftarrow{[\pi'_1,\pi'_2]}.\]
Thus the left hand side of (\ref{functor}) equals to its  right hand side. This finishes the proof.
\end{proof}

\begin{Rm}
In \cite{BL}, the authors constructed a strict Lie $2$-algebra on the multiplicative $1$-vector fields and their natural transformations. They proved the Morita equivalence of this construction and obtained a strict Lie $2$-algebra structure on the geometric stack. Actually, this is our case for $k=1$ when writing the strict Lie $2$-algebra  as a Lie algebra crossed module.
Another remark is that our graded Lie $2$-algebra is actually the same with the one in \cite{BCLX}, where they wrote it in the $2$-term $L_\infty$-algebra form $\Gamma(\wedge^\bullet A)\to \mathcal{X}^\bullet_{\mathrm{mult}} (\mathcal{G})$. Here we see affine multivector fields as the geometric support of this graded Lie $2$-algebra structure.
\end{Rm}

 Now we move to consider the infinitesimal of affine $k$-vector fields.

For an affine $k$-vector field $\Pi\in \mathfrak{X}^k_{\aff}(\mathcal{G})$, by Lemma \ref{affine vf},   define $\delta_\Pi f\in \Gamma(\wedge^{k-1} A)$ and $\delta_\Pi X\in \Gamma(\wedge^k A)$ for any $f\in C^\infty(M)$ and $X\in \Gamma(A)$, such that
\begin{eqnarray}\label{inf}
 \overrightarrow{\delta_\Pi f}=[\Pi,t^*f],\qquad \overrightarrow{\delta_\Pi X}=[\Pi,\overrightarrow{X}].
 \end{eqnarray}

Recall that
a {\bf $k$-differential} \cite{Xu} on a Lie algebroid $A$ is a pair of maps
\[\delta_0:C^\infty(M)\to \Gamma(\wedge^{k-1} A),\qquad \delta_1: \Gamma(A)\to \Gamma(\wedge^k A),\]
satisfying
\[ \delta_0(fg)=\delta_0(f) g+f\delta_0(g),\qquad \delta_1(fX)=\delta_0(f) X+f\delta_1(X),\qquad \forall f,g\in C^\infty(M), X\in \Gamma(A),\]
and
\[\delta_1[X,Y]=[\delta_1(X),Y]+[X,\delta_1(Y)],\qquad X, Y\in \Gamma(A).\]

Denote by $\oplus_k\mathfrak{X}_{\aff}^k(\mathcal{G})$ ($\oplus_k\mathfrak{X}_{\mathrm{mult}}^k(\mathcal{G})$) and $\oplus_k\mathcal{A}_k$ the spaces of affine (multiplicative) vector fields on $\mathcal{G}$ and $k$-differentials on $A$. It is direct to check that they are graded Lie algebras with the Schouten bracket and the commutative Lie bracket.

With this notions, by (\ref{inf}), we have a map
\begin{eqnarray}\label{delta}
\delta:\oplus_k \mathfrak{X}_{\aff}^k(\mathcal{G})\to \oplus_k \mathcal{A}_k,\qquad \Pi\mapsto \delta_{\Pi}.
\end{eqnarray}
The universal lifting theorem says that \[\delta|_{\oplus_k \mathfrak{X}_{\mathrm{multi}}^k(\mathcal{G})}:\oplus_k \mathfrak{X}_{\mathrm{multi}}^k(\mathcal{G})\to   \oplus_k \mathcal{A}_k\]
is an isomorphism of graded Lie algebras when $\mathcal{G}$ is $s$-connected and $s$-simply connected \cite{Xu}.

As a direct consequence of  Proposition \ref{leftright}, we have the following isomorphism of graded Lie algebras.
\begin{Pro}\label{multiaffineiso}
We have an isomorphism
\[\oplus_k \mathfrak{X}^k_{\aff}(\mathcal{G})\to \oplus_k \mathfrak{X}^k_{\mathrm{mult}}(\mathcal{G})\triangleright (\oplus_k \Gamma(\wedge^k A)),\qquad \Pi\mapsto (\Pi-\overrightarrow{\pi},\pi),\qquad \pi=\mathrm{pr}_{\wedge^k A} \Pi|_M,\]
of graded Lie algebras, where the brackets on $\oplus_k \mathfrak{X}^k_{\aff}(\mathcal{G})$ and $\oplus_k \mathfrak{X}^k_{\mathrm{mult}}(\mathcal{G})$ are the Schouten bracket, the bracket on $\oplus_k \Gamma(\wedge^k A)$ is the graded Lie bracket induced by the Lie bracket on $A$, and the mixed bracket is
\[[\Gamma, \pi]=\delta_\Gamma(\pi)\in \Gamma(\wedge^{k+l-1} A),\qquad \Gamma\in \mathfrak{X}^k_{\mathrm{mult}}(\mathcal{G}), \pi\in \Gamma(\wedge^l A).\]
\end{Pro}

\begin{proof}
By Proposition \ref{leftright}, this map is an isomorphism of graded vector spaces whose inverse is $(\Gamma,\pi)\mapsto \Gamma+\overrightarrow{\pi}$. Identifying an element $\pi\in \Gamma(\wedge^k A)$ with the affine $k$-vector field $\overrightarrow{\pi}$, we see that the Lie bracket on the right hand side is actually induced by the Schouten bracket on the left hand side under the isomorphism. Hence the right hand side is a graded Lie algebra and this map is an isomorphism of graded Lie algebras. We could also check directly this map is a morphism of Lie algebras. The key is to notice that
\begin{eqnarray}\label{Schouten bracket}
\mathrm{pr}_{\wedge^{k+l-1} A} [\Pi,\Pi']=\delta_\Pi (\pi')+(-1)^{kl} \delta_{\Pi'} (\pi)-[\pi,\pi'],\qquad \Pi\in \mathfrak{X}^k_\aff(\mathcal{G}),\Pi'\in \mathfrak{X}^l_\aff (\mathcal{G}).
\end{eqnarray}
The proof of this is similar to that in the proof of Theorem \ref{Lie 2} for the right translation instead of the left translation.
\end{proof}

The map $\delta$ defined in (\ref{delta}) is not a bijection on $\oplus_k \mathfrak{X}_{\aff}^k(\mathcal{G})$. In fact, for $\Pi\in \mathfrak{X}^k_{\aff}(\mathcal{G})$, we have
\begin{eqnarray}\label{inf bia}
\delta_{\Pi-\overrightarrow{\pi}}=\delta_\Pi-[\pi,\cdot],\qquad \delta_{\Pi-\overleftarrow{\pi}}=\delta_\Pi,\qquad \pi=\mathrm{pr}_{\wedge^k A} \Pi|_M.
\end{eqnarray}

Proposition \ref{multiaffineiso} together with the universal lifting theorem for multiplicative multivector fields tells us that  the kernel of the map $\delta$ is $\oplus_k \Gamma(\wedge^k A)$ and  we obtain the universal lifting theorem for affine multivector fields.

\begin{Thm}
Let $\mathcal{G}$ be  an $s$-simply connected and $s$-connected Lie groupoid with Lie algebroid $A$. We have a graded Lie algebra isomorphism
\[\oplus_k \mathfrak{X}_{\aff}^k(\mathcal{G})\cong \oplus_k \mathcal{A}_k\triangleright (\oplus_k\Gamma(\wedge^k A)),\qquad
\Pi\mapsto(\delta_{\Pi}-[\pi,\cdot], \pi),\qquad \pi=\mathrm{pr}_{\wedge^k A}\Pi|_M,\]
where the bracket on $\oplus_k \mathfrak{X}_{\aff}^k(\mathcal{G})$ and $\oplus_k \mathcal{A}_k$ are the Schouten bracket and the commutator bracket, the bracket on $\oplus_k \Gamma(\wedge^k A)$ is the graded Lie bracket induced by the Lie bracket on $A$ and the mixed bracket is
\[[\delta, \pi]=\delta(\pi)\in \Gamma(\wedge^{k+l-1} A),\qquad \delta\in \mathcal{A}_k,\pi\in \Gamma(\wedge^l A).\]
Here $\delta$ acts on $\pi$ as a degree $k-1$ derivation.
\end{Thm}
\begin{proof}
It follows from (\ref{inf bia}), Proposition \ref{multiaffineiso} and the universal lifting theorem for multiplicative multivector fields \cite{Xu}.\end{proof}
Next, we consider the case when an affine bivector field $\Pi$ on a Lie groupoid $\mathcal{G}$ is also Poisson. We shall generalize Lu's results for Lie groups in \cite{Lu}.

\begin{Pro}\label{Pi_r Poisson}
Let $\Pi$ be an affine  bivector field on a Lie groupoid $\mathcal{G}$, and $\Pi_r,\Pi_l$ be the multiplicative bivector fields given by \eqref{affmul}. Then
\begin{itemize}
\item[\rm(1)] $\Pi_r$ ($\Pi_l$) is Poisson if and only if $[\Pi,\Pi]$ is right (left)-invariant;
\item[\rm(2)]  if $\Pi_r$ is Poisson, then $\Pi$ is Poisson if and only if $2\delta_{\Pi_r} \pi+[\pi,\pi]=0$, where $\pi=\mathrm{pr}_{\wedge^2 A} \Pi|_M$.
\end{itemize}
\end{Pro}
\begin{proof}
For (1), direct calculation shows that
\[[\Pi_r,\Pi_r]=[\Pi-\overrightarrow{\pi},\Pi-\overrightarrow{\pi}]=[\Pi,\Pi]-2[\Pi,\overrightarrow{\pi}]+[\overrightarrow{\pi},\overrightarrow{\pi}]=[\Pi,\Pi]-2\overrightarrow{\delta_{\Pi} \pi} +\overrightarrow{[\pi,\pi]}. \]
So if $\Pi_r$ is Poisson, $[\Pi,\Pi]$ is right-invariant. Conversely, since by (\ref{Schouten bracket}),  we have $\mathrm{pr}_{\wedge^3 A}[\Pi,\Pi]|_M=2\delta_\Pi \pi-[\pi,\pi]$.  If $[\Pi,\Pi]$ is right-invariant, we must have $[\Pi,\Pi]=\overrightarrow{2\delta_\Pi \pi-[\pi,\pi]}$ and hence $[\Pi_r,\Pi_r]=0$.

For (2),  following from
\[[\Pi,\Pi]=[\Pi_r+\overrightarrow{\pi},\Pi_r+\overrightarrow{\pi}]=[\Pi_r,\Pi_r]+2\overrightarrow{\delta_{\Pi_r} \pi}+\overrightarrow{[\pi,\pi]},\]
we get the result.
\end{proof}
As a corollary, if an affine bi-vector field $\Pi$  is Poisson, the associated two multiplicative vector fields $\Pi_r$ and $\Pi_l$ are  also Poisson.
\begin{Cor}
Let $\Pi$ be an affine Poisson structure on a Lie groupoid $\mathcal{G}$ with $\pi=\mathrm{pr}_{\wedge^2 A} \Pi|_M$. Then its inverse as introduced in \eqref{Pi inverse}
\[\Pi^{-1}=\Pi-(\overrightarrow{\pi}+\overleftarrow{\pi})\]
 is also an affine Poisson structure on $\mathcal{G}$.
 \end{Cor}
\begin{proof}
$\Pi^{-1}$ is obviously affine. To see that  it is Poisson, we have
\begin{eqnarray*}
[\Pi^{-1},\Pi^{-1}]&=&[\Pi-(\overrightarrow{\pi}+\overleftarrow{\pi}),\Pi-(\overrightarrow{\pi}+\overleftarrow{\pi})]\\ &=&-2[\Pi,\overrightarrow{\pi}]-2[\Pi,\overleftarrow{\pi}]+[\overleftarrow{\pi},\overleftarrow{\pi}]+[\overrightarrow{\pi},\overrightarrow{\pi}]\\ &=&[\Pi_r,\Pi_r]+[\Pi_l,\Pi_l].
\end{eqnarray*}
Therefore, by Proposition \ref{Pi_r Poisson},  $\Pi^{-1}$ is Poisson.
\end{proof}

\begin{Ex}
Let $\mathcal{G}$ be a Lie groupoid with Lie algebroid $A$. For $\pi\in \Gamma(\wedge^2 A)$, the bivector field $\Pi=\overrightarrow{\pi}$ is affine. It is Poisson if and only if $\pi$ satisfies the classical Yang-Baxter equation $[\pi,\pi]=0$. Moreover, we have $\Pi_r=0, \Pi_l=\overrightarrow{\pi}-\overleftarrow{\pi}$ and $\Pi^{-1}=-\overleftarrow{\pi}$.

Besides, given any $\gamma\in \Gamma(\wedge^2 A)$, define $\Pi=\overrightarrow{\pi}+\overleftarrow{\gamma}$. Then we get $\Pi_r=\overleftarrow{\gamma}-\overrightarrow{\gamma}$ and $\Pi_l=\overrightarrow{\pi}-\overleftarrow{\pi}$ and $\Pi^{-1}=-\overleftarrow{\pi}-\overrightarrow{\gamma}$.
Furthermore, direct calculation shows that $\Pi$ is Poisson if and only if $\overrightarrow{[\pi,\pi]}=\overleftarrow{[\gamma,\gamma]}$, which implies that $[\pi,\pi]=[\gamma,\gamma]\in \mathrm{ker} \rho$ and both of them are $\Ad$-invariant.
\end{Ex}

Affine Poisson structures give rise to a natural equivalence relation between multiplicative Poisson structures on a Lie groupoid (Poisson groupoids), which further gives an equivalence relation on Lie bialgebroids.

\section{Affine $k$-forms on a Lie groupoid}

A $k$-form $\Theta\in \Omega^k(\mathcal{G})$  on a Lie groupoid $\mathcal{G}$ is multiplicative if the graph of multiplication $\{(g,h,gh);s(g)=t(h)\}$, or space of triangles, is an isotropic submanifold of $\mathcal{G}\times \mathcal{G}\times \mathcal{G}$ with respect to $\Theta\oplus \Theta\oplus -\Theta$. Algebraically, a $k$-form $\Theta$ on $\mathcal{G}\rightrightarrows M$ is multiplicative \cite{BC2,BC,C} if it satisfies that
\[m^*\Theta=\mathrm{pr}_1^*\Theta+\mathrm{pr}_2^*\Theta,\]
where $m,\mathrm{pr}_1,\mathrm{pr}_2:\mathcal{G}^{(2)}\to \mathcal{G}$ are the groupoid multiplication and the projections to the first and second components respectively.

One consequence of the multiplicativity condition is that $\Theta$ is isotropic on $M$, that is $\Theta(X_1,\cdots,X_k)=0$ for any $X_1,\cdots,X_k\in \mathfrak{X}^1(M)$. Namely, $\iota^*\Theta=0$, where $\iota:M\hookrightarrow \mathcal{G}$ is the natural inclusion. In other words, the restriction of $\Theta$ on $M$ has no component in $\wedge^k T^*M$.  Relaxing this condition, we shall get the notion of affine $k$-forms.

The restriction of a $k$-form  $\Theta\in \Omega^k(\mathcal{G})$ on $M$ has $k+1$-components: \[\Theta|_M\in \Gamma(\wedge^k T^*\mathcal{G}|_M)=\Gamma(\wedge^k(A^*\oplus T^*M))=\Gamma(\wedge^k A^*\oplus \wedge^{k-1} A^*\wedge T^*M\wedge\cdots \wedge \wedge^k T^*M).\] Denote by $\theta$ the $\wedge^k T^*M$-component:
$\theta=\mathrm{pr}_{\wedge^k T^*M} \Theta|_M$. In another words, $\theta=\iota^*\Theta$ for $\iota:M\hookrightarrow \mathcal{G}$.
 \begin{Def}\label{aff form}
A $k$-form $\Theta\in \Omega^k(\mathcal{G})$ on a Lie groupoid $\mathcal{G}$ is affine  if it satisfies
\begin{eqnarray}\label{affine form}
m^*\Theta=\mathrm{pr}_1^*\Theta+\mathrm{pr}_2^*\Theta-\mathrm{pr_1}^*s^*\theta,
\end{eqnarray}
where $\theta:=\mathrm{pr}_{\wedge^k T^*M} \Theta|_M$.
\end{Def}
Since $s\circ \mathrm{pr}_1=t\circ \mathrm{pr}_2: \mathcal{G}^{(2)}\to \mathcal{G}$, the affine condition has another expression
\begin{eqnarray}\label{affine form 2}
m^*\Theta=\mathrm{pr}_1^*\Theta+\mathrm{pr}_2^*\Theta-\mathrm{pr_2}^*t^*\theta.
\end{eqnarray}
One direct consequence of the definition is that the de Rham differential of an affine $k$-form on $\mathcal{G}$ is an affine $k+1$-form.

Unlike the multiplicative case, it is not obvious from (\ref{affine form}) that a $k$-form is affine if the submanifold of parallelograms is isotropic in $\mathcal{G}\times \mathcal{G}\times \mathcal{G}\times \mathcal{G}$.

\begin{Pro}
A $k$-form $\Theta$ on $\mathcal{G}$ is affine if and only if the space of parallelograms \[\Gamma=\{(g,h,l,hg^{-1}l);s(g)=s(h),t(g)=t(l)\}\] is an isotropic submanifold of
$\mathcal{G}\times \mathcal{G}\times \mathcal{G}\times \mathcal{G}$ with respect to $\Theta\oplus -\Theta\oplus -\Theta\oplus \Theta$, that is
\begin{eqnarray}\label{iso}
i^*(\mathrm{pr}_1^* \Theta-\mathrm{pr}_2^* \Theta-\mathrm{pr}_3^*\Theta+\mathrm{pr}_4^*\Theta)=0,
\end{eqnarray}
where $\mathrm{pr}_i: \mathcal{G}\times \mathcal{G}\times \mathcal{G}\times \mathcal{G}\to \mathcal{G}$ is the projection to the $i$-th component and $i:\Gamma\hookrightarrow
\mathcal{G}\times \mathcal{G}\times \mathcal{G}\times \mathcal{G}$ is the inclusion.
\end{Pro}
\begin{proof}
The tangent space of $\Gamma$ at $(g,h,l,hg^{-1}l)$ consists of  4-tuples $(X_g,Y_h, Z_l, Y_h\cdot \mathrm{inv}(X)_{g^{-1}}\cdot Z_l)$
of tangent vectors, where $Y_h\cdot \mathrm{inv}(X)_{g^{-1}}\cdot Z_l$ means the multiplication of three tangent vectors in $T\mathcal{G}$. Applying (\ref{iso}) to $k$ such vectors, we have
\begin{eqnarray}\label{triple}
\nonumber&&\Theta(Y^1_h\cdot \mathrm{inv}(X)^1_{g^{-1}}\cdot Z^1_l,\cdots,Y_h^k\cdot \mathrm{inv}(X)^k_{g^{-1}}\cdot  Z^k_l)\\ &=&
-\Theta(X_g^1,\cdots,X_g^k)+\Theta(Y_h^1,\cdots,Y_h^k)+\Theta(Z_l^1,\cdots,Z_l^k).
\end{eqnarray}
In particular, we choose $(h,l)\in \mathcal{G}^{(2)}$ and $g=1_{t(l)}=1_{s(h)}$. Moreover, $(X^i_g,Y^i_h,Z^i_l)$ are chosen to satisfy $t(Z^i_l)=s(Y^i_h)=X^i_g$. Then the equation becomes
\begin{eqnarray}\label{double}
\Theta(Y^1_h\cdot Z^1_l,\cdots,Y_h^k\cdot  Z^k_l)=
-\Theta(s(Y_h^1),\cdots,s(Y_h^k))+\Theta(Y_h^1,\cdots,Y_h^k)+\Theta(Z_l^1,\cdots,Z_l^k).
\end{eqnarray}
This is exactly (\ref{affine form}).

Conversely, if $\Theta$ is affine, by setting $l=h^{-1}$ and $Z_l^i=\mathrm{inv}(Y)_l^i$ in (\ref{double}), we get
\begin{eqnarray}\label{inverse}
\Theta(\mathrm{inv}(Y)_l^1,\cdots,\mathrm{inv}(Y)_l^k)=\Theta(s(Y_h^1),\cdots,s(Y_h^k))-\Theta(Y_h^1,\cdots,Y_h^k)+\Theta(t( Y_h^1),\cdots,t(Y_h^{k})).
\end{eqnarray}
Applying (\ref{double}) twice to the left hand side of (\ref{triple}), we obtain
\begin{eqnarray*}
&&\Theta(Y^1_h\cdot \mathrm{inv}(X)^1_{g^{-1}}\cdot Z^1_l,\cdots,Y_h^k\cdot \mathrm{inv}(X)^k_{g^{-1}}\cdot  Z^k_l)\\ &=&
\Theta(Y_h^1,\cdots,Y_h^k)+\Theta(\mathrm{inv}(X)^1_{g^{-1}}\cdot  Z^1_l,\cdots,\mathrm{inv}(X)^k_{g^{-1}}\cdot  Z^k_l)-\Theta(s(X_g^1),\cdots,s(X_g^k))\\ &=&\Theta(Y_h^1,\cdots,Y_h^k)+\Theta(\mathrm{inv}(X)^1_{g^{-1}},\cdots,\mathrm{inv}(X)^1_{g^{-1}})+\Theta(Z_l^1,\cdots,Z_l^k)\\ &&-\Theta(t(X^1_g),\cdots,t(X_g^k))-\Theta(s(X_g^1),\cdots,s(X_g^k))\\ &=&
\Theta(Y_h^1,\cdots,Y_h^k)-\Theta(X_g^1,\cdots,X_g^k)+\Theta(Z_l^1,\cdots,Z_l^k),
\end{eqnarray*}
where we have used (\ref{inverse}) in the last step. Hence, we get (\ref{iso}).
\end{proof}

Regarding to the relation between multiplicative and affine $k$-forms, we have already seen that a multiplicative $k$-form is an affine $k$-form which is isotropic on $M$. On the other hand,  an affine $k$-form is associated with two multiplicative $k$-forms.

\begin{Pro}\label{aff mul form}
Let  $\Theta\in \Omega^k(\mathcal{G})$ be a $k$-form on $\mathcal{G}$ with $\theta=\mathrm{pr}_{\wedge^k T^*M} \Theta|_M\in \Omega^k(M)$.
Define two $k$-forms on $\mathcal{G}$:
\[\Theta_l:=\Theta-s^*\theta,\qquad \Theta_r:=\Theta-t^*\theta.\]
Then $\Theta$ is affine if and only if $\Theta_l$ ($\Theta_r$) is a multiplicative $k$-form.
\end{Pro}
\begin{proof}
By straightforward calculation, we have
\[m^*\Theta_l=m^*\Theta-m^*s^*\theta,\qquad \mathrm{pr}_1^*\Theta_l=\mathrm{pr}_1^*\Theta-\mathrm{pr}_1^*s^*\theta,\qquad
\mathrm{pr}_2^*\Theta_l=\mathrm{pr}_2^*\Theta-\mathrm{pr}_2^*s^*\theta.\]
Following from $s\circ \mathrm{pr_2}=s\circ m:\mathcal{G}^{(2)}\to \mathcal{G}$, we see that the equation
$m^*\Theta_l=\mathrm{pr}_1^*\Theta_l+\mathrm{pr}_2^*\Theta_l$ holds if and only if (\ref{affine form}) holds.
Similarly, noticing that $t\circ \mathrm{pr}_1=t\circ m$, we get that $\Theta_r$ is multiplicative if and only if $\Theta$ satisfies
(\ref{affine form 2}), that is, $\Theta$ is affine.
\end{proof}

Denote by $\Omega_{\mathrm{aff}}^k(\mathcal{G})$ and $\Omega_{\mathrm{mult}}^k(\mathcal{G})$ the spaces of affine and multiplicative $k$-forms respectively. It is immediate that $\Omega_{\mathrm{aff}}^k(\mathcal{G})$ is a vector space with $\Omega_{\mathrm{mult}}^k(\mathcal{G})$ being a linear subspace.
\begin{Thm}\label{2-group for form}
We have a $2$-vector space
\[\Omega^k_{\mathrm{aff}}(\mathcal{G})\rightrightarrows \Omega^k_{\mathrm{mult}}(\mathcal{G}),\]
where the groupoid structure is given as follows: the source and target maps are
\[s(\Theta)=\Theta_r,\qquad t(\Theta)=\Theta_l,\qquad \forall \Theta\in \Omega^k_{\mathrm{aff}}(\mathcal{G}),\]
and the multiplication is
\[\Theta* \Theta'=\Theta+s^*\theta',\qquad \theta'=\mathrm{pr}_{\wedge^k T^*M} \Theta'|_M\]
for a pair $\Theta,\Theta'\in \Omega^k_{\mathrm{aff}}(\mathcal{G})$ such that $\Theta_r=\Theta'_l$.
\end{Thm}

\begin{proof}
The proof is similar to that for Theorem \ref{2-group for vf}.
\end{proof}

\begin{Cor}
The $2$-term chain complex of vector spaces associated to the above $2$-vector space $\Omega^k_{\mathrm{aff}}(\mathcal{G})\rightrightarrows \Omega^k_{\mathrm{mult}}(\mathcal{G})$ is
\[\Omega^k(M)\to \Omega^k_{\mathrm{mult}}(\mathcal{G}),\qquad \theta\mapsto t^*\theta-s^*\theta.\]
\end{Cor}

It is seen from the definition that the affine and multiplicative forms are closed under the de Rham differential. So we get two subcomplexes of the de Rham complex on $\mathcal{G}$:\[(\Omega^\bullet_{\mathrm{mult}}(\mathcal{G}),d)\subset (\Omega^\bullet_{\aff}(\mathcal{G}),d)\subset (\Omega^\bullet(\mathcal{G}),d).\]
\begin{Pro}\label{4.6}
The map
\[\Phi: \Omega^\bullet_{\aff}(\mathcal{G})\to \Omega^\bullet_{\mathrm{mult}}(\mathcal{G})\oplus \Omega^\bullet(M),\qquad \Theta\mapsto (\Theta-t^*\theta,\theta),\quad \theta=\mathrm{pr}_{\wedge^k T^*M} \Theta|_M\]
is an isomorphism of cochain complexes, where the differentials are the de Rham differential. Thus we get an isomorphism on the cohomology
\[H^\bullet_{\aff}(\mathcal{G})\cong H_{\mathrm{mult}}^\bullet(\mathcal{G})\oplus H^\bullet(M).\]
\end{Pro}

\begin{proof}
The inverse of $\Phi$ can be defined by $(\Lambda,\lambda)\mapsto \Lambda+t^*\lambda$ for any $\Lambda\in \Omega^k_{\mathrm{mult}}(\mathcal{G})$ and $\lambda\in \Omega^k(M)$. So $\Phi$ is an isomorphism.  Next, we check that it is a cochain map, namely, $d\circ \Phi=\Phi\circ d$. Since $\theta=\iota^*\Theta$ for $\iota:M\hookrightarrow \mathcal{G}$,  we have $d\theta=\mathrm{pr}_{\wedge^{k+1} T^*M} d\Theta$. Then we have
\[d\circ \Phi(\Theta)=(d\Theta-t^*d \theta,d\theta)=\Phi\circ d(\Theta).\]
Thus it induces an isomorphism on the cohomology.
\end{proof}

Now we discuss the infinitesimal of affine $k$-forms. It is known from \cite{BC,C} that there is a one-to-one correspondence between multiplicative $k$-forms on $\mathcal{G}$ and $IM$ $k$-forms on its Lie algebroid $A$ when $\mathcal{G}$ is $s$-connected and $s$-simply connected.

A pair $(\mu,\nu)$ of bundle maps
\[\mu:A\to \wedge^{k-1} T^*M,\qquad \nu:A\to \wedge^k T^*M,\qquad k\geq 1\]
is called an {\bf IM k-form} on $A$ if it satisfies that
\begin{eqnarray*}
\iota_{\rho(X)} \mu(Y)&=&-\iota_{\rho(Y)} \mu(X),\\
\mu([X,Y])&=&\mathcal{L}_{\rho(X)} \mu(Y)-\iota_{\rho(Y)} d\mu(X)-\iota_{\rho(Y)} \nu(X),\\
\nu([X,Y])&=&\mathcal{L}_{\rho(X)}\nu(Y)-\iota_{\rho(Y)} d\nu(X)\qquad \forall X,Y\in \Gamma(A).
\end{eqnarray*}
The pair $(\mu,\nu)$ determines a linear $k$-form on the vector bundle $A$.  We can describe  these conditions in the way that
the induced map $\oplus_A^k TA\to \mathbbm{R}$ is a Lie algebroid morphism with the tangent Lie algebroid structure on $\oplus_A^k TA\to \oplus_A^k TM$ and the trivial Lie algebroid structure on $\mathbbm{R}\to 0$. See \cite{BC} for details.

%\begin{Thm}\cite{BC}
%Let $\mathcal{G}$ be a source simply-connected Lie groupoid over $M$ with Lie algebroid $A$. Then there is a one-one correspondence between multiplicative $k$-forms on $\mathcal{G}$ and IM $k$-forms on $A$: \[\Omega^k_{mult}\to \Omega_{IM}^k(A),\qquad \alpha\mapsto (\mu,\nu)\] where
%\[\langle \mu(a),X_1\wedge\cdots\wedge X_{k-1}\rangle=\alpha(\overrightarrow{a},X_1,\cdots,X_{k-1}),\qquad \langle \nu(a),X_1\wedge\cdots\wedge X_k\rangle=d\alpha(\overrightarrow{a},X_1,\cdots,X_k),\]
%for $a\in \Gamma(A)$ and $X_1,\cdots,X_k\in \mathfrak{X}(M)$.
%\end{Thm}
By Proposition \ref{4.6},  the infinitesimal corresponding of affine $k$-forms on Lie groupoid $\mathcal{G}$ is clear.
\begin{Pro}
There is a one-one correspondence between affine $k$-forms $\Theta$ and triples $(\mu,\nu,\theta)$ of an IM-form $(\mu,\nu)$ and $\theta\in \Omega^k(M)$. That is
\[\Omega^k_{\aff}(\mathcal{G})\cong \Omega_{\mathrm{mult}}^k(\mathcal{G})\oplus \Omega^k(M)\cong \Omega_{\mathrm{IM}}^k(A)\oplus \Omega^k(M),\]
\[\Theta\mapsto(\Theta_r:=\Theta-t^*\theta,\theta:=\mathrm{pr}_{\wedge^k T^*M} \Theta|_M)\mapsto (\mu,\nu, \theta).\]
\end{Pro}

%\begin{Rm}
%Another description of a multiplicative $k$-form on a Lie groupoid $\mathcal{G}$ is that it defines a morphism of Lie groupoids $\Pi:( \oplus^{k-1} T\mathcal{G}\rightrightarrows \oplus^{k-1} TM) \to (T^*\mathcal{G}\rightrightarrows A^*)$.
%An affine structure differs from a multiplicative structure by a $k$-form on $M$, which gives a map $\oplus^{k-1} TM\to T^*M$.
%Actually, if we treat multiplicative $k$-forms as functors, an affine structure is a natural transformation between the related two multiplicative structures. So the affine structure is the categorification of multiplicative structures.
%\end{Rm}

\begin{Ex} For a Lie groupoid $\mathcal{G}\rightrightarrows M$, given any $\theta\in \Omega^k(M)$, then $s^*\theta$ and $t^*\theta$ are affine $k$-forms on $\mathcal{G}$ and $s^*\theta-t^*\theta$ is a multiplicative $k$-form on $\mathcal{G}$.
\end{Ex}
\begin{Ex}
On a Lie group $G$, affine $k$-forms are multiplicative $k$-forms. They are nonzero only when $k$ is $0$ and $1$.
This is because for any $k\geq 2$, \[\Theta((X_1, 0,X_3,\cdots,X_k)\cdot(0,Y_2,Y_3,\cdots,Y_k))=\Theta(R_{h_1} X_1,L_{g_2} Y_2, X_3\cdot Y_3,\cdots,X_k \cdot Y_k)=0\] for $X_i\in T_{g_i} G, Y_j\in T_{h_j} G$. So multiplicative $1$-forms on a Lie group are always closed.

On the abelian group $\mathbbm{R}^n$, multiplicative $1$-forms are constant $1$-forms. In local coordinates, they have the form $\Theta=\sum_i c_i dx^i$, where $c_i$ is a constant.
\end{Ex}

\begin{Ex}

For the pair Lie groupoid $M\times M\rightrightarrows M$, multiplicative $k$-forms all have the form $\mathrm{pr}_1^* \alpha-\mathrm{pr}_2^* \alpha$ for $\alpha\in \Omega^k(M)$, where $\mathrm{pr}_i:M\times M\to M$ is the projection to the $i$-th component. And affine $k$-forms are of the form $\pr_1^*\alpha+\pr_2^* \beta$ for any two $k$-forms $\alpha,\beta\in \Omega^k(M)$.

%This is followed from the fact that the IM $k$-forms on $TM$ are just pairs of $(\alpha,d\alpha)$ with $\alpha\in \Omega^k(M)$.
\end{Ex}

\section{Affine tensors on a Lie groupoid}

\subsection{Definition of affine tensors}
We shall define a tensor to be affine if it is affine when seen as a function on a more complicated Lie groupoid. This is    motivated by the notion of  multiplicative $(p,q)$-tensors introduced in \cite{BD}.

Affine functions on a Lie groupoid $\mathcal{G}\rightrightarrows M$ are naturally defined as affine $0$-forms on $\mathcal{G}$.

\begin{Def}
A function $F\in C^\infty(\mathcal{G})$ is affine if it satisfies
\begin{eqnarray}\label{affine function}
F(gh)=F(g)+F(h)-F(s(g)),\qquad \forall (g,h)\in \mathcal{G}^{(2)},
\end{eqnarray}
 or $F(gh)=F(g)+F(h)-F(t(h))$ as $s(g)=t(h)$.
 \end{Def}
 In particular, if an affine function $F$ satisfies $F|_M=0$,  it is called a multiplicative function. By the definition, the space of affine functions is a vector space with the space of multiplicative functions as a subspace.

 As a corollary of Proposition \ref{aff mul form} for $0$-forms, we have

 \begin{Lem}\label{tt}
 Let $F\in C^\infty(\mathcal{G})$ be a function on a Lie groupoid $\mathcal{G}\rightrightarrows M$ and $f=\iota^* F\in C^\infty(M)$  the restriction of $F$ on $M$, where $\iota: M\hookrightarrow \mathcal{G}$ is the natural inclusion. Define
 \[F_l=F-s^*f,\qquad F_r=F-t^*f.\]
 Then $F$ is affine if and only if $F_l$ or $F_r$ is a multiplicative function on $\mathcal{G}$.
 \end{Lem}
 %\begin{proof}
% By direct calculation, if $F$ is affine, then
% \[F_l(gh)=F(gh)-f(s(h))=F(g)+F(h)-f(s(g))-f(s(h))=F_l(g)+F_l(h).\]
 %Thus $F_l$ is multiplicative. The other direction follows from the same calculation.
 %\end{proof}
 \begin{Ex}\label{ex for affine functions}
 For any function $f\in C^\infty(M)$, we see that $s^* f-t^* f\in C^\infty(\mathcal{G})$ is a multiplicative function on $\mathcal{G}$ and $s^*f, t^*f$ are affine functions on $\mathcal{G}$.
  \end{Ex}
Consider the Lie groupoid
\[\Gamma: \oplus^q T\mathcal{G}\oplus^p T^*\mathcal{G}\rightrightarrows \oplus^q TM \oplus^p A^*.\]
A $(p,q)$-tensor $F\in \Gamma(\wedge^p T\mathcal{G}\otimes \wedge^q T^*\mathcal{G})$ on $\mathcal{G}$ can be viewed as a function on $\Gamma$.
\begin{Def}
A $(p,q)$-tensor  $F\in \Gamma(\wedge^p T\mathcal{G}\otimes \wedge^q T^*\mathcal{G})$ on a Lie groupoid $\mathcal{G}$ is called affine if it is an affine function on $\Gamma$.
\end{Def}

 The following proposition ensures the consistence of this definition with that in Definition \ref{aff vf} and \ref{aff form} for the cases of affine $k$-vector fields and affine $k$-forms.
\begin{Pro}
\begin{itemize}
\item[\rm(1)]
An affine $(p,0)$-tensor is an affine $p$-vector field  as defined in Definition \ref{aff vf};
\item[\rm(2)]  An affine $(0,q)$-tensor is an affine $q$-form as defined in Definition  \ref{aff form}.
\end{itemize}
\end{Pro}
\begin{proof}

Let $F\in \Gamma(\wedge^p T\mathcal{G})$ be an affine $(p,0)$-tensor. Namely,
\[F(\xi^1_g\cdot \eta^1_h,\cdots, \xi^p_g\cdot \eta^p_h)=F(\xi_g^1,\cdots,\xi_g^p)+F(\eta_h^1,\cdots,\eta_h^p)-F(s(\xi^1_g),\cdots,s(\xi^p_g)),\]
for $\xi_g^i\in T_g^* \mathcal{G}$ and $\eta_h^i\in T_h^*\mathcal{G}$ such that $s(\xi_g^i)=t(\eta_h^i)$, where $s,t$ are the source and target maps in the Lie groupoid $T^*\mathcal{G}\rightrightarrows A^*$.

Let $f=\mathrm{pr}_{\wedge^p A} F|_M$,  the projection of $F$ restricting on $M$ to $\wedge^k A$. We claim that
$F$ is an affine $(p,0)$-tensor if and only if $F-\overrightarrow{f}$ is a multiplicative $(p,0)$-tensor, that is
\[(F-\overrightarrow{f})(\xi^1_g\cdot \eta^1_h,\cdots, \xi^p_g\cdot \eta^p_h)=(F-\overrightarrow{f})(\xi_g^1,\cdots,\xi_g^p)+(F-\overrightarrow{f})(\eta_h^1,\cdots,\eta_h^p).\]
By (\ref{s cot}), we have $t^*f=\overrightarrow{f}$, where $t$ is the target map in $T^*\mathcal{G}\rightrightarrows A^*$.
the assertion holds by Lemma \ref{tt}.

By \cite[Proposition 2.7]{Xu}, $F-\overrightarrow{f}$ is a multiplicative function if and only if it is a multiplicative $p$-vector field on $\mathcal{G}$, which is further equivalent to that $F$ is an affine $p$-vector field by Proposition \ref{leftright}. So $F$ is an affine $(p,0)$-tensor if and only if $F$ is an affine $p$-vector field.

%From \cite{Xu28}, $(-\mu,L_{\mathcal{X}}^*\mu,L_{\mathcal{Y}}^*\mu,-L_{\mathcal{X}}^*R_{\mathcal{Y}}^*\mu)$ is conormal to $S$ at the point $(gh,h,g,s(h))$ for $\mu\in T_{gh}^*\mathcal{G}$. Thus, if $F$ is an affine $p$-vector field, we have
%\begin{eqnarray}\label{mu}
%F(\mu)=F(L_{\mathcal{X}}^*\mu)+F(R_{\mathcal{Y}}^*\mu)-F(L_{\mathcal{X}}^*R_{\mathcal{Y}}^*\mu).
%\end{eqnarray}
%Now let $\mu=\xi_g\cdot \eta_h$ such that $s(\xi_g)=t(\eta_h)$. By straightforward calculation, we get
%\[L_{\mathcal{X}}^*(\xi_g\cdot \eta_h)=(\mathcal{X}\circ t)^*\xi+\eta,\qquad R_{\mathcal{Y}}^*(\xi_g\cdot \eta_h)=\xi+(Y\circ (t\circ Y)^{-1}\circ s)^*\eta,\]
%and
%\[L_{\mathcal{X}}^*R_{\mathcal{Y}}^*(\xi_g\cdot \eta_h)\]

If $F\in \Gamma(\wedge^q T^*\mathcal{G})$ is an affine $(0,q)$-tensor, then
\begin{eqnarray*}\label{0q}
F(X\cdot Y)=F(X)+F(Y)-F(s(X)),\quad X\in \wedge^q T_g \mathcal{G}, Y\in \wedge^q T_h \mathcal{G}, (g,h)\in \mathcal{G}^{(2)}, s(X)=t(Y).
\end{eqnarray*}
%By definition, if we write $X=X_1\wedge\cdots\wedge X_q\in \wedge^q T_g \mathcal{G}$ and $Y=Y_1\wedge\cdots\wedge Y_q\in \wedge^q T_h \mathcal{G}$, then we have $X\cdot Y=Tm(X_1,Y_1)\wedge \cdots \wedge Tm(X_q,Y_q)$, where $Tm: T_g \mathcal{G}\times T_h \mathcal{G}\to T_{gh} \mathcal{G}$ is the differential of the multiplication on $\mathcal{G}$.
which implies that
\[m^* F=\mathrm{pr}_1^* F+\mathrm{pr}_2^* F-\mathrm{pr}_1^*s^*F.\]
So $F$ is an affine $q$-form as defined in Definition \ref{aff form}.
\end{proof}

Regarding to the relation between affine and multiplicative $(p,q)$-tensors, we also have the assertion as for affine $k$-vector fields and affine $k$-forms.

Let $f\in \Gamma(\wedge^p A\otimes \wedge^q T^*M)$. View it as a function on the base manifold of the Lie groupoid
\[\Gamma:  \oplus^q T\mathcal{G}\oplus^p T^*\mathcal{G}\rightrightarrows \oplus^q TM \oplus^p A^*.\]
By Example \ref{ex for affine functions},  $s_\Gamma^*f$ and $t_\Gamma^* f$ are affine functions on $\Gamma$ and hence affine $(p,q)$-tensors on $\mathcal{G}$, where $s_\Gamma,t_\Gamma$ are the source and target maps of the Lie groupoid $\Gamma$.

\begin{Lem}
Let $f\in \Gamma(\wedge^p A\otimes \wedge^q T^*M)$. Denote by $\overleftarrow{f}:=s_{\Gamma}^*f,\overrightarrow{f}:=t^*_\Gamma f\in \Gamma(\wedge^p T\mathcal{G}\otimes \wedge^q T^* \mathcal{G})$.
We have
\begin{eqnarray}\label{fleft}
\overleftarrow{f}(X_1,\cdots,X_q)=\overleftarrow{f(sX_1,\cdots, sX_q)},\qquad \overrightarrow{f}(X_1,\cdots,X_q)=\overrightarrow{f(tX_1,\cdots, tX_q)},
\end{eqnarray}
for $(X_1,\cdots, X_q)\in  T\mathcal{G}^{(q)}$.
\end{Lem}
\begin{proof}
This follows from the definition and (\ref{s cot}).
\end{proof}
If assuming $f=u\otimes \beta$ for $u\in \Gamma(\wedge^p A)$ and $\beta\in \Omega^q(M)$, we get
\[\overleftarrow{f}=\overleftarrow{u}\otimes s^*\beta,\qquad \overrightarrow{f}=\overrightarrow{u}\otimes t^*\beta.\]

The following result is a direct consequence of Lemma \ref{tt}.

\begin{Pro}\label{affine-mult}
Let $F\in \Gamma(\wedge^p T\mathcal{G}\otimes \wedge^q T^* \mathcal{G})$ and $f=\mathrm{pr}_{\wedge^p A\otimes \wedge^q T^*M} F|_M$.
 Define
\[F_l=F-\overleftarrow{f},\qquad  F_r=F-\overrightarrow{f},\]
where $\overleftarrow{f}$ and $\overrightarrow{f}$ are defined in \eqref{fleft}.
Then $F$ is an affine $(p,q)$-tensor on $\mathcal{G}$ if and only if $F_l$ or $F_r$ is a multiplicative $(p,q)$-tensor.
\end{Pro}

Denote by $T^{p,q}_{\mathrm{aff}} (\mathcal{G})$ and $T^{p,q}_{\mathrm{mult}} (\mathcal{G})$ the spaces of affine and multiplicative $(p,q)$-tensors on $\mathcal{G}$ respectively. It is immediate that $T^{p,q}_{\aff}(\mathcal{G})$ is a vector space with $T^{p,q}_{\mathrm{mult}}(\mathcal{G})$ being a linear subspace.
\begin{Thm}\label{Ver1}
With the above notations, we have a $2$-vector space
\[T^{p,q}_{\mathrm{aff}} (\mathcal{G})\rightrightarrows T^{p,q}_{\mathrm{mult}} (\mathcal{G}),\]
where the source and target maps of the groupoid structure are
\[s(F)=F_r,\qquad t(F)=F_l,\qquad \forall F\in T^{p,q}_{\mathrm{aff}} (\mathcal{G})\]
and for a pair $F_1,F_2\in T^{p,q}_{\mathrm{aff}} (\mathcal{G})$ such that $(F_1)_r=(F_2)_l$, the multiplication is
\[F_1* F_2=F_1+\overleftarrow{f_2},\qquad f_2=\mathrm{pr}_{\wedge^p A\otimes \wedge^q T^*M} F_2|_M.\]
\end{Thm}
\begin{proof}
The proof is similar to that for Theorem \ref{2-group for vf}. \end{proof}
\begin{Cor}
The $2$-term complex of vector spaces of the above $2$-vector space is
\[\Gamma(\wedge^p A\otimes \wedge^q T^*M)\to T^{p,q}_{\mathrm{mult}} (\mathcal{G}),\qquad f\mapsto \overrightarrow{f}-\overleftarrow{f}.\]
\end{Cor}

\subsection{The strict monoidal category of affine $(1,1)$-tensors}
Another important case is affine $(1,1)$-tensors, which can be used to define affine Nijenhuis operators on a Lie groupoid.
On the space of affine $(1,1)$-tensors, in addition to the $2$-vector space structure proposed in Theorem \ref{Ver1}, we shall also construct a strict monoidal category structure in this subsection.

A $(1,1)$-tensor on a Lie groupoid $\mathcal{G}$ is {\bf multiplicative} if the induced bundle map $T\mathcal{G}\to T\mathcal{G}$ is a Lie groupoid morphism \cite{XuC}, which amounts to say that the corresponding function on  $T^*\mathcal{G}\oplus T\mathcal{G}\rightrightarrows A^*\oplus TM$ is multiplicative by \cite[Proposition 3.9]{BD}. Thus the composition of two multiplicative $(1,1)$-tensors is still a multiplicative $(1,1)$-tensor. We shall show that the composition of two affine $(1,1)$-tensors is also an affine $(1,1)$-tensor.

By  (\ref{affine function}),
a $(1,1)$-tensor $N\in \Gamma(T\mathcal{G}\otimes T^*\mathcal{G})$ is affine if it satisfies that
\[N(X\cdot Y,\xi\cdot \eta)=N(X,\xi)+N(Y,\eta)-n(s_{T\mathcal{G}}X,s_{T^*\mathcal{G}}\xi),\qquad n=\mathrm{pr}_{T^*M\otimes A} N|_M\]
where $(X,Y)\in T\mathcal{G}^{(2)}$ and $(\xi,\eta)\in T^*\mathcal{G}^{(2)}$ are multiplicable pairs in $T\mathcal{G}$ and $T^*\mathcal{G}$  covering the same pair $(g,h)\in \mathcal{G}^{(2)}$. Here $s_{T\mathcal{G}}X$ and $s_{T^*\mathcal{G}}$ are the source maps of the tangent and cotangent groupoids respectively.

For a multiplicative $(1,1)$-tensor $N$, it corresponds to a Lie groupoid morphism $(N,n_{TM}):T\mathcal{G}\to T\mathcal{G}$, where $n_{TM}:TM\to TM$ is the map on the base manifold. Since $N$ preserves the $s$-fiber, it induces a bundle map $n_A:A\to A$. Then we have $N|_M=n_{TM}+n_A$.  For an affine $(1,1)$-tensor $N$,  from the difference of affine and multiplicative $(1,1)$-tensors, we have that the restriction of $N$ on $M$ is
\[N|_M=\begin{pmatrix} n_{TM}&  0 \\ n & n_A\end{pmatrix}: TM\oplus A\to TM\oplus A,\]
where $n_{TM}:TM\to TM$ and $n_A:A\to A$ and $n=\mathrm{pr}_{T^*M\otimes A} N|_M$.
\begin{Lem}\label{Hor1}
The composition $N\circ N'$ of two affine $(1,1)$-tensors $N$ and $N'$ is still an affine $(1,1)$-tensor with \[(N\circ N')_l=N_l\circ N'_l,\qquad (N\circ  N')_r=N_r\circ N'_r,\qquad \forall N,N'\in T_{\aff}^{1,1}(\mathcal{G}).\]

Moreover, the  $T^*M\otimes A$-component of $N\circ N'|_M$ is
\[\mathrm{pr}_{T^*M \otimes A} N\circ N'|_M=n_A\circ n'+n\circ n'_{TM}+n\circ \rho\circ n',\]
where
\[N|_M=\begin{pmatrix} n_{TM}&  0 \\ n & n_A\end{pmatrix},\quad N'|_M=\begin{pmatrix} n'_{TM}&  0 \\ n' & n'_A\end{pmatrix}: TM\oplus A\to TM\oplus A\]
are the decompositions of $N$ and $N'$ restricting on $M$ and $\rho:A\to TM$ is the anchor map.
\end{Lem}
\begin{proof}
Write $N=N_r+\overrightarrow{n}$ and $N'=N'_r+\overrightarrow{n'}$, where $n$ and $n'$ are the $T^*M\otimes A$-components of  $N|_M$ and $N'|_M$ respectively and $\overrightarrow{n},\overrightarrow{n'}$ are defined in (\ref{fleft}). Then
\[N\circ N'=N_r\circ N'_r+N_r\circ \overrightarrow{n'}+\overrightarrow{n}\circ N'_r+\overrightarrow{n}\circ \overrightarrow{n'}.\]
By Proposition \ref{affine-mult}, $(N_r,n_{TM}),(N'_r,n'_{TM}): T\mathcal{G}\to T\mathcal{G}$ are morphisms of Lie groupoids. So we have the formulas $N_r(\overrightarrow{u})=\overrightarrow{n_A(u)}, \forall u\in \Gamma(A)$ and $t\circ N'_r=n'_{TM}\circ t$.
Applying on $X\in T_g \mathcal{G}$, we get
\[ N_r\circ \overrightarrow{n'}(X)=N_r\overrightarrow{n'(tX)}=\overrightarrow{n_A(n'(tX))}=\overrightarrow{n_A\circ n'}(X),\]
\[ \overrightarrow{n}\circ N'_r(X)= \overrightarrow{n(tN'_r(X))}=\overrightarrow{n(n'_{TM}(tX))}=\overrightarrow{n\circ n'_{TM}}(X),\]
and
\[\overrightarrow{n}\circ \overrightarrow{n'}(X)=\overrightarrow{n}\overrightarrow{n'(tX)}=\overrightarrow{n(\rho(n'(tX)))}=\overrightarrow{n\circ \rho\circ n'}(X).\]
Thus we proved that
\[N\circ N'=N_r\circ N'_r+\overrightarrow{n_A\circ n'+n\circ n'_{TM}+n\circ \rho\circ n'}.\]
Notice that the composition $N_r\circ N'_r$ of two multiplicative $(1,1)$-tensors is still multiplicative. Then by Proposition \ref{affine-mult}, we obtain that  $N\circ N'$ is an affine $(1,1)$-tensor with the properties as desired.
\end{proof}

Actually, the $2$-vector space $T^{(1,1)}_{\mathrm{aff}}(\mathcal{G})\rightrightarrows T^{(1,1)}_{\mathrm{mult}}(\mathcal{G})$ from Theorem \ref{Ver1} with the composition is a strict monoidal category.

\begin{Thm}
We have a strict monoidal category structure on the $2$-vector space
\[T^{(1,1)}_{\mathrm{aff}}(\mathcal{G})\rightrightarrows T^{(1,1)}_{\mathrm{mult}}(\mathcal{G}),\]
with the product being the composition of two affine $(1,1)$-tensors and the unit given by the identity $I:T\mathcal{G}\to T\mathcal{G}$.
%where
%\begin{itemize}
%\item{} the source, target maps are $s(N)=N_r, t(N)=N_l,\forall N\in T^{(1,1)}_{\mathrm{aff}}(\mathcal{G})$ and the multiplication is
%\[N_1* N_2=N_1+\overleftarrow{n_2},\qquad n_2=\mathrm{pr}_{T^*M\otimes A} N_2|_M,\]
%for two affine $(1,1)$-tensors $N_1,N_2$  satisfying $(N_1)_r=(N_2)_l$;
%\item{} the product is the composition of two affine $(1,1)$-tensors.
%\end{itemize}
\end{Thm}

\begin{proof}
It is obvious that the identity $I:T\mathcal{G}\to T\mathcal{G}$, as a multiplicative $(1,1)$-tensor, is a left and right unit for the composition.

It suffices to verify that the composition $\circ: T_{\aff}^{1,1}(\mathcal{G})\times T_{\aff}^{1,1}(\mathcal{G})\to T_{\aff}^{1,1}(\mathcal{G})$ is a bifunctor.
 Let $N_1, N_2, N_3$ and $ N_4$ be four affine $(1,1)$-tensors such that $(N_1)_r=(N_3)_l$ and $(N_2)_r=(N_4)_l$. That is $N_1-\overrightarrow{n_1}=N_3-\overleftarrow{n_3}$ and $N_2-\overrightarrow{n_2}=N_4-\overleftarrow{n_4}$.

By Lemma \ref{Hor1}, we see $(N_1\circ N_2)_r=(N_3\circ N_4)_l$. Then we prove
\begin{eqnarray}\label{eqc}
(N_1 * N_3)\circ(N_2 * N_4)=(N_1 \circ N_2)*(N_3\circ N_4).
\end{eqnarray}
The left hand side equals to
\[(N_1+\overleftarrow{n_3})\circ (N_2+\overleftarrow{n_4})=N_1\circ N_2+N_1\circ \overleftarrow{n_4}+ \overleftarrow{n_3}\circ N_2+\overleftarrow{n_3}\circ \overleftarrow{n_4},\]
and by Lemma \ref{Hor1}, the right hand side amounts to
\[N_1\circ N_2+\overleftarrow{ ({n_3}_{A}\circ n_4+n_3\circ {n_4}_{TM}+n_3\circ \rho\circ n_4)}.\]
By the same calculation in Lemma \ref{Hor1} for the left translation instead of the right translation, we get
\[N_1\circ \overleftarrow{n_4}=(N_3-\overleftarrow{n_3}+\overrightarrow{n_1})\circ \overleftarrow{n_4}=\overleftarrow{{n_3}_A\circ n_4}+\overrightarrow{n_1}\circ \overleftarrow{n_4}=\overleftarrow{{n_3}_A\circ n_4},\]
which follows from
\[\overrightarrow{n_1}\circ \overleftarrow{n_4}(X)=\overrightarrow{n_1}(\overleftarrow{n_4(sX)}=\overrightarrow{n_1(t\overleftarrow{n_4(sX))}}=0,\qquad X\in T_g \mathcal{G}.\] Similarly, we have
\[\overleftarrow{n_3}\circ N_2=\overleftarrow{n_3\circ {n_4}_{TM}}.\]
Observe that $\overleftarrow{n_3}\circ \overleftarrow{n_4}=\overleftarrow{n_3\circ \rho\circ n_4}$ for the reason as for the right translation proved in Lemma \ref{Hor1}.
This proves (\ref{eqc}).
\end{proof}
\begin{Rm}
This strict monoidal category from affine $(1,1)$-tensors is related to the $2$-vector spaces constructed from affine $2$-vector fields and $2$-forms as in Theorem \ref{2-group for vf} and \ref{2-group for form} if we take into consideration of the generalized tangent bundle $T\mathcal{G}\oplus T^*\mathcal{G}\rightrightarrows TM\oplus A^*$. Identify an affine $2$-vector field $\Pi\in \mathfrak{X}^k(\mathcal{G})$ with a matrix
\[
\left(
\begin{array}{cc}
I& \Pi \\
0 & I
\end{array}
\right):T\mathcal{G}\oplus T^*\mathcal{G}\to T\mathcal{G}\oplus T^*\mathcal{G}.\]
Then the addition in the vector space $\mathfrak{X}^2_{\aff}(\mathcal{G})$ is actually the composition of two affine $2$-vector fields as matrices. Likewise, seeing an affine $2$-form $\Theta\in \Omega^2(\mathcal{G})$ as a matrix \[
\left(
\begin{array}{cc}
I& 0 \\
\Theta & I
\end{array}
\right):T\mathcal{G}\oplus T^*\mathcal{G}\to T\mathcal{G}\oplus T^*\mathcal{G},\]
we get that the addition in $\Omega^2_{\aff}(\mathcal{G})$ is the composition of  two affine $2$-forms as  matrices.
\end{Rm}

\begin{Rm}
As multiplicative $(1,1)$-tensors can be characterized as a Lie groupoid morphism from $T\mathcal{G}$ to $T\mathcal{G}$, a
multiplicative $p$-vector field defines a morphism of Lie groupoids from $\oplus^{p-1} T^*\mathcal{G}$ to $T\mathcal{G}$ and a multiplicative $p$-form defines a morphism of Lie groupoids from $\oplus^{p-1} T\mathcal{G}$ to $T^* \mathcal{G}$. See for example \cite{KS}.
In these cases, viewing multiplicative structures as functors, we see that affine structures are actually natural transformations between these multiplicative structures.
\end{Rm}

\begin{Ex}
One example of multiplicative $(1,1)$-tensors on a Lie groupoid $\mathcal{G}$ is $\overrightarrow{n}-\overleftarrow{n}$ for any $n\in \Gamma(T^*M\otimes A)$. And $\overrightarrow{n}$ and $\overleftarrow{n}$ are both affine $(1,1)$-tensors on $\mathcal{G}$.

\end{Ex}

\begin{Ex}\label{Ex:group}
An affine $(1,1)$-tensor on a Lie group $G$ is a multiplicative $(1,1)$-tensor on $G$, which is a $G$-equivariant linear map from $\mathfrak{g}:=\mathrm{Lie}(G)$ to $\mathfrak{g}$. Namely,
\[T^{1,1}_{\mathrm{aff}}(G)=\{N\in \End(\mathfrak{g})|N(\Ad_g u)=\Ad_gN(u), g\in G,u\in \mathfrak{g}\}.\]

The product  of two affine $(1,1)$-tensors in the strict monoidal category structure is the composition of linear maps and the groupoid multiplication is trivial, i.e. $N* N=N$, meaning that an affine $(1,1)$-tensor can only multiply itself, which results in itself.
\end{Ex}

\begin{Ex}\label{Ex:pair}
For the pair groupoid $\mathcal{G}=M\times M\rightrightarrows M$, a multiplicative $(1,1)$-tensor on $\mathcal{G}$ is always of the form $\overrightarrow{N}- \overleftarrow{N}$ for a bundle map $N:TM\to TM$, where by definition,
\[(\overrightarrow{N}-\overleftarrow{N})(X,Y)=\overrightarrow{N(X)}-\overleftarrow{N(Y)}=(N(X),0)+(0,N(Y))=(N(X),N(Y)), \]
for all $(X,Y)\in T_x M\times T_y M$.
In fact, for any $u\in \mathfrak{X}(M)$, we get $\overrightarrow{u}(x,y)=\frac{d}{dt}|_{t=0} (\phi_t^u(x),x)(x,y)=(u,0)$, where $\phi_t^u(x)$ is a flow of $u$ such that $\phi_0^u(x)=x$. And $\overleftarrow{u}(x,y)=-\frac{d}{dt}|_{t=0}(x,y)(y,\phi_t^u(y))=-(0,u)$, where the minus sigh comes from the convention that $\overleftarrow{u}(\cdot)=-L_{(\cdot)} \mathrm{inv}(u)$.

By the relation between affine and multiplicative  $(1,1)$-tensors, an affine $(1,1)$-tensor is of the form $\overrightarrow{N}+\overleftarrow{N'}$ for two bundle maps $N,N'\in \End(TM)$. For simplicity, we write an affine $(1,1)$-tensor  as $(N,N')$. The product in the strict monoidal category structure  is the composition of two $(1,1)$-tensors:
\[(N_1,N_2)\circ (N_3,N_4)=(N_1\circ N_3,-N_2\circ N_4),\qquad \forall N_i\in \End(TM),\]
which follows from
\[(\overrightarrow{N_1}+\overleftarrow{N_2})\circ  (\overrightarrow{N_3}+\overleftarrow{N_4})(X,Y)=(\overrightarrow{N_1}+\overleftarrow{N_2})(N_3(X),-N_4(Y))=(N_1\circ N_3(X),N_2\circ N_4(Y)).\]
For the groupoid multiplication, two affine $(1,1)$-tensors $(N_1,N_2)$ and $(N_3,N_4)$ are multiplicable if and only if $N_2=-N_3$ and the multiplication is
\[(N_1,N_2)*(N_3,N_4)=\overrightarrow{N_1}+\overleftarrow{N_2}+\overleftarrow{N_3}+\overleftarrow{N_4}=\overrightarrow{N_1}+\overleftarrow{N_4}=(N_1,N_4).\]
\end{Ex}
The next example we are interested in is the direct sum of the pair groupoid and a Lie group: $M\times M\times G\rightrightarrows M$. The following proposition tells us that  this case is just the direct sum of Example \ref{Ex:group} and \ref{Ex:pair}. So the strict monoidal category structure for this case is also clear.

\begin{Pro}
An affine $(1,1)$-tensor $N$ on $M\times M\times G\rightrightarrows M$ is of the form
\[N(X,Y,g,u)=(N_1(X),N_2(Y),g,L(u)),\qquad \forall N_1,N_2\in \End(TM), L\in  \End(\mathfrak{g})^G,\]
for $X\in T_x M, Y\in T_y M, g\in G$ and $u\in \mathfrak{g}$. It is multiplicative if and only if $N_1=N_2$.
\end{Pro}
\begin{proof}
 A multiplicable $(1,1)$-tensor field on $\mathcal{G}$ is a bundle map \[N=(N_1,N_2,N_3):TM\times TM\times TG\to TM\times TM\times TG\]
over the base manifold $M\times M\times G$ and also a groupoid morphism over $TM$ to itself. We claim that a multiplicative $(1,1)$-tensor field can only be of the form
\[N(X,Y,g,u)=(N(X),N(Y),g,L(u)),\qquad \forall X\in T_x M,Y\in T_y M,g\in G,u\in \mathfrak{g},\]
for some $N\in \End(TM)$ and a $G$-equivariant linear map $L\in \End(\mathfrak{g})^G$.
In fact, since $N$ preserves the $s$ and $t$-fibers, the first and second components $N_1, N_2$ have to be the same.
 Then  $N$ being a groupoid morphism requires that
\[N_3(X,Y,g,u)+\Ad_g N_3(Y,Z,h,v)=N_3(X,Z,gh,u+\Ad_g v).\]
Since $N_3$ is a bundle map, it follows that $N_3(X,Y,g,u)$ is independent of $X$ and $Y$ and it is determined by a $G$-equivariant linear map $L\in \End(\mathfrak{g})$.  We thus easily get that an affine $(1,1)$-tensor field $N$ on $M\times M\times G\rightrightarrows M$ is of the form
\[N(X,Y,g,u)=(N_1(X),N_2(Y),g,L(u)),\qquad \forall N_1,N_2\in \End(TM),L\in  \End(\mathfrak{g}).\]
Hence the strict monoidal structure for this case is clear.
\end{proof}

At the end of this section, we provide a class of affine $(1,1)$-tensors coming from the composition of affine $2$-vector fields and affine $2$-forms.

\begin{Pro}
The composition $\Pi\circ \Theta:T\mathcal{G}\to T\mathcal{G}$ of an affine $2$-vector field $\Pi\in \mathfrak{X}^2(\mathcal{G})$ and an affine $2$-form $\Theta\in \Omega^2(\mathcal{G})$  is an affine  $(1,1)$-tensor with
\begin{eqnarray}\label{eqe}
\mathrm{pr}_{T^*M\otimes A}(\Pi\circ \Theta)|_M=\pi_{A^*} \circ \theta+\pi\circ \theta_{TM}+\pi\circ \rho^*\circ \theta,
\end{eqnarray}
where $\pi\in \Gamma(\wedge^2 A), \pi_{A^*}\in \Gamma(A\otimes T^*M), \theta_{TM}\in \Gamma(T^*M\otimes A^*)$ and $\theta\in \Omega^2(M)$ are the corresponding components of $\Pi$ and $\Theta$ restricting on $M$ respectively.

Moreover, the associated two multiplicative $(1,1)$-tensors are
\[(\Pi\circ \Theta)_l=\Pi_l\circ \Theta_l,\qquad (\Pi\circ \Theta)_r=\Pi_r\circ \Theta_r.\]
\end{Pro}

\begin{proof}
Denote by $\Pi=\Pi_r+\overrightarrow{\pi}$ and $\Theta=\Theta_r+t^*\theta$, where $\Pi_r$ and $\Theta_r$ are the associated multiplicative $2$-vector field and $2$-form. Then
\begin{eqnarray}\label{eqd}
\Pi\circ \Theta=\Pi_r\circ \Theta_r+\overrightarrow{\pi}\circ \Theta_r+\Pi_r\circ t^*\theta+\overrightarrow{\pi}\circ t^*\theta.
\end{eqnarray}
Acting on $X\in T_g \mathcal{G}$ and pairing with $\xi\in T_g^*\mathcal{G}$, we find
\[\langle \overrightarrow{\pi}\circ \Theta_r(X),\xi\rangle=\langle\pi, R_g^* \Theta_r(X)\wedge R_g^* \xi\rangle=\langle \pi, t\Theta_r(X)\wedge t\xi\rangle=\langle\pi, \theta_{TM}(tX)\wedge t\xi\rangle=\langle \overrightarrow{\pi\circ \theta_{TM}(tX)},\xi\rangle,\]
where we have used (\ref{s cot}), i.e. $t_{T^*\mathcal{G}}=R_g^*: T_g^*\mathcal{G}\to A_{t(g)}^*$ and the fact that $(\Theta_r, \theta_{TM})$ is a Lie groupoid morphism from $T\mathcal{G}$ to $T^*\mathcal{G}$.
This implies that \[\overrightarrow{\pi}\circ \Theta_r=\overrightarrow{\pi\circ \theta_{TM}}.\] On the other hand,
using $t_{T^*\mathcal{G}}=R_g^*$ again and the fact that $(\Pi_r,\pi_{A^*}):T^*\mathcal{G}\to T\mathcal{G}$ is a Lie groupoid morphism,
we have \[\langle\Pi_r\circ t^*\theta(X),\xi\rangle=-\langle \theta, tX\wedge t\Pi_r(\xi)\rangle=-\langle \theta(tX), \pi_{A^*}t(\xi)\rangle=\langle \overrightarrow{\pi_{A^*}\circ \theta(tX)}, \xi\rangle.\]
Thus \[\Pi_r\circ t^*\theta=\overrightarrow{\pi_{A^*}\circ \theta}.\] At the end, observe that $t\overrightarrow{\pi}(\xi)=\rho\pi(t\xi)$ since
\[\langle t\overrightarrow{\pi}(\xi),\alpha\rangle=\langle \pi,R_g^*\xi\wedge R_g^*t^*\alpha\rangle=\langle \pi(t\xi),t^*\alpha\rangle=\langle \pi(t\xi),\rho^*\alpha\rangle, \qquad \forall \alpha\in \Omega^1(M).\]
We obtain
\[\langle \overrightarrow{\pi}\circ t^*\theta(X),\xi\rangle=-\langle \theta(tX),t\overrightarrow{\pi}(\xi)\rangle=-\langle \theta(tX),\rho\pi(t\xi)\rangle=\langle \overrightarrow{\pi\circ \rho^*\circ \theta(tX)},\xi\rangle.\]
Hence
\[\overrightarrow{\pi}\circ t^*\theta=\overrightarrow{\pi\circ \rho^*\circ \theta}.\]
Coupled with (\ref{eqd}), we get
\[\Pi\circ \Theta=\Pi_r\circ \Theta_r+\overrightarrow{(\pi\circ \theta_{TM}+\pi_{A^*}\circ \theta+\pi\circ \rho^*\circ \theta)}.\]
Since $\Pi_r\circ \Theta_r: T\mathcal{G}\to T\mathcal{G}$ is a Lie groupoid morphism and thus gives a multiplicative $(1,1)$-tensor, we proved that $\Pi\circ \Theta$ is an affine $(1,1)$-tensor. The other assertions are also clear.
\end{proof}

\end{document}